\documentclass[a4paper,12pt]{amsart}
\usepackage[english]{babel}
\usepackage[utf8]{inputenc} 
\usepackage{amssymb,amsthm,amsmath} 
\usepackage{indentfirst}
\usepackage[makeroom]{cancel}
\usepackage{microtype}
\usepackage{color}
\usepackage{enumerate} 
\usepackage{tikz} 
\usetikzlibrary{cd} 
\usepackage{bbm} 
\usepackage{mathdots}
\usepackage[hidelinks,draft=false]{hyperref}
\usepackage[colorinlistoftodos,textsize=small,textwidth=80]{todonotes}
\usepackage[toc]{appendix}
\usepackage{mathtools}
\usepackage{cite}
\usepackage{stmaryrd} 
\usepackage[left=3cm, right=3cm, top=3cm, bottom=3cm]{geometry}
\usepackage{mathabx,epsfig}
\usepackage{quiver}

\usepackage[all]{xy}

\hypersetup{
    linktoc=all,     
}

\makeatletter
\DeclareFontEncoding{LS1}{}{}
\DeclareFontSubstitution{LS1}{stix}{m}{n}
\DeclareMathAlphabet{\mathscr}{LS1}{stixscr}{m}{n}
\makeatother

\newtheorem{theorem}{Theorem}[section]
\newtheorem{lemma}[theorem]{Lemma}
\newtheorem*{lemma*}{Lemma}

\newtheorem{proposition}[theorem]{Proposition}

\theoremstyle{definition}
\newtheorem{definition}[theorem]{Definition}

\theoremstyle{remark}
\newtheorem{remark}[theorem]{Remark}
\newtheorem{example}[theorem]{Example}

\newcommand{\B}{\mathbf{B}}

\newcommand{\R}{\mathbb{R}}
\newcommand{\Z}{\mathbb{Z}}

\newcommand{\SSS}{\mathbb{S}}

\newcommand{\lrb}[1]{\left\lbrace#1\right\rbrace}

\renewcommand{\phi}{\varphi}

\newcommand{\Bord}{\mathrm{Bord}}

\newcommand{\hofib}[1]{\mathrm{hofib}\left(#1\right)}
\newcommand{\hofiblax}[1]{\mathrm{hofib}_{\mathrm{lax}}\left(#1\right)}

\newcommand{\conn}[1]{#1_\nabla}

\newsavebox{\pullback}
\sbox\pullback{%
\begin{tikzpicture}%
\draw (0,0) -- (1ex,0ex);%
\draw (1ex,0ex) -- (1ex,1ex);%
\end{tikzpicture}}

\title{Integrals detecting degree 3 string cobordism classes}

\author{Domenico Fiorenza}
\address{Sapienza Universit\`a di Roma; Dipartimento di Matematica ``Guido Castelnuovo'', P.le Aldo Moro, 5 - 00185 - Roma, Italy; 
}
\email{fiorenza@mat.uniroma1.it}
\author{Eugenio Landi}
\address{Sapienza Universit\`a di Roma; Dipartimento di Matematica ``Guido Castelnuovo'', P.le Aldo Moro, 5 - 00185 - Roma, Italy; }
\email{eugenio.landi@uniroma1.it}
\begin{document}

\begin{abstract}
    {The third string bordism group $\Bord_3^{\mathrm{String}}$ is known to be $\mathbb{Z}/24\mathbb{Z}$. Using Waldorf's notion of a geometric string structure on a manifold, Bunke--Naumann and Redden have exhibited integral formulas involving the Chern-Weil form representative of the first Pontryagin class and the canonical 3-form of a geometric string structure that realize the isomorphism $\Bord_3^{\mathrm{String}} \to \mathbb{Z}/24\mathbb{Z}$. We will show how these formulas naturally emerge when one considers certain natural $\mathrm{U}(1)$-valued and $\R$-valued 3d TQFT associated with the classifying stacks of Spin bundles with connection and of String bundles with geometric structure, respectively.}
\end{abstract}

\maketitle

\section{Introduction}
{
It is not hard to show that in degree $n$ less or equal to $6$ the $\mathrm{String}$ bordism group, $\Bord_n^{\mathrm{String}}$ is isomorphic to the framed bordism  group $\Bord^\mathrm{fr}_n$. Indeed, $\mathrm{BString}=\mathrm{BO}\langle 8\rangle$ is the eighth stage in the Whitehead tower of the classifying space of orthogonal group and so the obstructions to lifting the classifying map of the tangent bundle of an $n$-dimensional string manifold $M_n$ through all the higher stages of the tower lie in the cohomology groups $H^k(M_n;\pi_k\mathrm{BO})$ for $k\geq8$. In particular, all of these obstructions vanish if $n\leq 7$. By the celebrated Pontryagin-Thom isomorphism one has $\Bord^\mathrm{fr}_n\cong \pi_n(\SSS)$, where $\SSS$ is the sphere spectrum or, equivalently, that $\Bord^\mathrm{fr}_n$ is isomorphic to the $n$-th stable homotopy group of the spheres. All this together, for $n=3$, gives 
\[
 \Bord_3^{\mathrm{String}}\cong \Bord_3^{\mathrm{fr}}\cong \pi_3(\SSS)\cong \Z/24\mathbb{Z}
\]
One may wish to express the isomorphism $\varphi\colon\Bord_3^{\mathrm{String}}\xrightarrow{\cong}\Z/24\Z$ as some characteristic number given by integrating some \emph{canonical} differential 3-form on a closed string 3-manifold $M_3$: $\phi[M_3]=\int_{M_3}\omega_{3;M_3}$. Clearly, there is no hope that this can be true, since the integral takes real values while $\phi$ takes values in $\Z/24\Z$, and there is no injective group homomorphism from $\Z/24\Z$ to $\mathbb{R}$. There is however a variant of this construction that may work. Instead of considering just a string 3-manifold $M_3$, one considers a string 3-manifold $M_3$ endowed with some additional structure $\Upsilon$. This structure should be such that any $M_3$ admits at least one $\Upsilon$. To the pair $(M_3,\Upsilon)$ there could be associated a canonical 3-form  $\omega_{3;M_3,\Upsilon}$ such that $\int_{M_3}\omega_{3;M_3,\Upsilon}$ takes integral values. Then, if a change in the additional structure $\Upsilon$ results in a change in the value $\int_{M_3}\omega_{3;M_3,\Upsilon}$ by a multiple of 24 one would have a well defined element
\[
\int_{M_3}\omega_{3;M_3,\Upsilon} \mod 24
\]
in $\Z/24\Z$, 
depending only on the string 3-manifold $M_3$; and this could indeed represent the isomorphism $\varphi$. In this form the statement is indeed almost true. The correct version of it has been found by Bunke and Naumann \cite{bunke-naumann} and, independently, by Redden \cite{redden}. Their additional datum $\Upsilon$ consists of a triple $(\eta^{}_{M_3},W_4,\nabla^{}_{W_4})$, where $\eta_{M_3}^{}$ is a geometric string structure on $M_3$ in the sense of Waldorf \cite{waldorf}, $W_4$ is a spin 4-manifold with $\partial W_4=M_3$,\footnote{Such a $W_4$ surely exists, since $\Bord_3^{\mathrm{Spin}}=0$.} and $\nabla^{}_{W_4}$ is a spin connection on $W_4$ such that the restriction $\nabla^{}_{W_4}\bigr\vert_{M_3}$ coincides with the spin connection datum of the geometric string structure $\eta_{M_3}$. Out of the data $(M_3,\eta^{}_{M_3},W_4,\nabla^{}){W_4})$ one can compute 
\[
\psi(M_3,\eta^{}_{M_3},W_4,\nabla^{}_{W_4}):=\frac{1}{2}\int_{W_4} \mathbf{p}_1^{\mathrm{CW}}(\nabla^{}_{W_4})-\int_{M_3} \omega^\eta_{M_3},
\]
where $\mathbf{p}_1^{\mathrm{CW}}(\nabla^{}_{W_4})$ is the Chern-Weil 4-form for the first Pontryagin class, evaluated on the connection $\nabla^{}_{W_4}$, and $\omega^\eta_{M_3}$ is the canonical 3-form associated with the geometric string structure $\eta^{}_{M_3}$.\footnote{This expression has been recently considered by Gaiotto--Johnson-Freyd--Witten in the context of minimally supersymmetric models in two dimensions \cite{gfw}. {\color{black}A twisted version, also crucially relying on the triviality of Spin bordism groups in dimensions $4k+3$ for $k\leq 8$, had previously been considered by Sati in \cite{sati-mbranes2010}}.}
From the interplay between geometric string structures and differential cohomology it follows that $\psi(M_3,\eta^{}_{M_3},W_4,\nabla^{}_{W_4})\in \Z$. Keeping $(M_3,\eta^{}_{M_3})$ fixed and letting $(W_4,\nabla^{}_{W_4})$ vary, one finds
\[
\psi(M_3,\eta^{}_{M_3},W_4^1,\nabla^{}_{W_4^1})-\psi(M_3,\eta^{}_{M_3},W_4^0,\nabla^{}_{W_4^0})=\frac{1}{2}\int_{W_4}p_1(W_4),
\]
where $W_4=W_4^1\cup_M (W_4^0)^{\mathrm{opp}}$ denotes the closed spin 4-manifold obtained gluing together $W_4^0$ (with the opposite orientation) and $W_4^1$ along $M_3$, and $\frac{1}{2}p_1(W_4)\in H^4(W_4;\Z)$ is the first fractional Pontryagin class of $W_4$. One has
\[
\hat{A}(W_4)=-\frac{1}{24}\int_{W_4}p_1(W_4)
\]
for any closed oriented 4-manifold $W_4$. By the Atiyah-Singer index theorem, the $\hat{A}$-genus of a closed oriented manifold is an integer if the manifold is spin, and is an even integer if moreover the dimension of the manifold is of the form $8k+4$. Therefore in our case we have that
$\int_{W_4}p_1(W_4)\in 48\Z$ and so $\int_{W_4}\frac{1}{2}p_1(W_4)\in 24\Z$. Therefore the function 
\[
\psi(M_3,\eta):=\psi(M_3,\eta^{}_{M_3},W_4,\nabla^{}_{W_4}) \mod 24
\]
is well defined. One concludes by showing that $\psi(M_3,\eta^{}_{M_3})$ is actually independent of the geometric string structure $\eta^{}_{M_3}$, and only depending on the string cobordism class of $M_3$. Additivity is manifest from the definition, so the above integral formula defines a group homomorphism $\psi\colon \Bord_3^{\mathrm{String}}\xrightarrow{\cong}\Z/24\Z$. A direct computation with the canonical generator of $\Bord_3^{\mathrm{String}}$, i.e., with $S^3$ endowed with the trivialization of its tangent bundle coming from  $S^3\cong \mathrm{SU}(2)$, then shows that $\psi$ is indeed an isomorphism.
\par
The aim of this note is to show how the above integral formula for $\psi$, as well as its main properties, naturally emerge in the context of topological field theories with values in the symmetric monoidal categories associated with morphisms of abelian groups. {\color{black}The topological field theories we will be dealing with will have ``background fields'' described by maps to some smooth stack, and one of the aims of this note is precisely to show how the use of this formalisms clarifies and organizes a few aspects of bordism invariants. In particular, we show how $\mathbb{R}$-valued bordism invariants can naturally be realized by integral formulas in the context of functorial field theories with background fields, and the example of $\Bord_3^{\mathrm{String}}$ shows these theories are flexible enough to also capture a few torsion phenomena (see also \cite{b-e:How-do-field-theories-detect-the-torsion} for closely related results). It is an interesting question whether every torsion bordism invariant admits a functorial field theory realization.}

{\color{black}
\subsection*{Notation and conventions}
Throughout the whole note, we will denote by $M_d$ a $d$-dimensional smooth manifold, and by $M$ a smooth manifold whose dimension is not relevant. For $G$ a Lie group, we will denote by $BG$ the classifying space of $G$, i.e., the topological space (unique up to homotopy) such that 
\[
[M,BG]\cong \{\text{principal $G$-bundles over $M$}\}/\text{isomorphism},
\]
where $[-,-]$ denotes the set of homotopy classes of maps, and by $\mathbf{B}G$ the smooth stack of principal $G$-bundles, i.e., the smooth stack associating with a smooth manifold $M$ the groupoid of principal $G$-bundles over $M$ with their isomorphisms. Clearly, $\mathbf{B}G$ retains finer information with respect to $BG$, which is recovered from $\mathbf{B}G$ by topological realization. For this reason, in the note we will always prefer working with the smooth stack $\mathbf{B}G$ rather than with the classifying space $BG$. The same consideration applies to all the constructions in the note, e.g., we will prefer working with the smooth stack $\mathbf{B}^3\mathrm{U}(1)$ over the Eilenberg-MacLane space $K(\mathbb{Z},4)$.

We will be assuming the reader has some familiarity with homotopy theory, higher categories, and stacks, but even a barely intuitive idea of what a smooth stack is, as sketched at the beginning of 
Section \ref{sec:cat-from-mor}, should suffice to follow all of the constructions in the note. We refer the reader to \cite{cosoneurs} for a comprehensive introduction to smooth stacks in the context of field theories, and to \cite{lurie-htt} for the general theory of higher categories and higher stacks.

All of the commutative diagrams in higher categories are to be read with given fillers, that are omitted for ease of notation. On a few occasions, e.g., in Remark \ref{rem:loop-space-torsor}, when the fillers are invertible we invert them without the notation (that omits the fillers at all) showing this. We hope this will cause no confusion to the reader.
The authors thank the Referee for their helpful comments and suggestions.
}
}

\section{Symmetric monoidal categories from morphisms of abelian groups and TQFTs}\label{sec:cat-from-mor}
{By $\Bord_{d,d-1}^\xi(X)$ we will denote the symmetric monoidal category of $(d,d-1)$-bordism with tangential structure $\xi$ and {\color{black}target} $X$. Contenting ourselves with an informal definition,\footnote{See, e.g., \cite{lurie, gradypavlov}  for a rigorous definition.} we mean that the objects of $\Bord_{d,d-1}^\xi(X)$ are $(d-1)$-dimensional closed manifolds $M_{d-1}$ equipped with a certain reduction $\xi$ of the structure group of the ``$d$-stabilized'' tangent bundle $TM\oplus \mathbb{R}$, and a map $f\colon M\to X$ to a  space, or more generally to a smooth stack, $X$. {\color{black} Not to leave this definition too abstract, let us recall that a smooth stack $X$ is by definition a sheaf of $\infty$-groupoids/Kan complexes over the site of smooth manifolds with open covers. This means that $X$ is a rule that associates with any smooth manifold $M$ an $\infty$-groupoid $X(M)$ in such a way that $X(M)$ can be pulled back along smooth maps $N\to M$ (i.e., $X$ is a contravariant functor), and such that $X(M)$ is completely and uniquely determined by its restrictions to the open sets of a cover $\{U_i\}$ of $M$, by the transition functions for these restrictions over double intersections, by the homotopies between transition functions on the triple overlaps, etc. (one says that ``the descent data are effective''). Classical examples are the 1-stack $\mathbf{B}G$, for $G$ a Lie group, associating a manifold $M$ with the groupoid $\mathbf{B}G(M)$ having principal $G$-bundles over $M$ as objects and isomorphisms of principal $G$ bundles as morphism, or the $0$-stacks $\Omega^{k}$ and $\Omega^{k}_{cl}$ associating with a smooth manifold $M$ the sets $\Omega^k(M)$ and $\Omega^k_{cl}(M)$ of smooth $k$-forms and smooth closed $k$-forms on $M$, respectively, seen as discrete groupoids (i.e., with only identities as morphisms).  Classical nonexamples are 
$\Omega^k_{\mathrm{ex}}$, since a compatible family of local exact $k$-forms does not necessarily define a globally exact $k$-form, or Riemannian metrics, since the pullback of a Riemannian metric along an arbitrary smooth map is not necessarily a nondegenerate pairing.

In the Physics' parlance, the map $f$ is called a ``background field'' of the theory; following this terminology, we will call $X$ ``the stack of background fields''. Notice that every smooth manifold $M$ can be seen as a smooth stack under the Yoneda identification of $M$ with the functor of smooth maps to $M$. Via this identification, objects of $X(M)$ are equivalently seen as morphisms of smooth stacks $M\to X$ and so as background fields; 1-morphisms in $X(M)$ then correspond to homotopies of maps $M\to X$ and so to gauge transformations of the background fields, and so on.} 

Morphisms $W_d\colon M_{d-1}^0\to M_{d-1}^1$ in   $\Bord_{d,d-1}^\xi(X)$ are $d$-manifolds with a $\xi$-structure on the tangent bundle and map to $X$ such that $\partial W_d=M_{d-1}^0\coprod M_{d-1}^1$ and such that the restrictions of the tangential structure and of the map to the target of $W_d$ coincide ``up to a sign'' with those of the $M_{d-1}^i$'s. The only tangential structures we will be concerned with will be orientations, spin, and string structures; we will denote them by or, Spin, and String, respectively. The monoidal structure on $\Bord_{d,d-1}^\xi(X)$ is given by disjoint union.

\begin{remark}
The trivial bundle $\mathbb{R}$ will always be oriented with $\{1\}$ as a positively oriented basis. This way the datum of a $d$-stable orientation on  $(d-1)$-manifold $M_{d-1}$ reduces to the datum of an orientation of $M_{d-1}$.  
\end{remark}

\begin{example}\label{ex:closed-forms}
Let $X=\Omega_{cl}^{d-1}$ be the smooth stack of closed $(d-1)$-forms. Then an object of  $\Bord_{d,d-1}^\mathrm{or}(\Omega_{cl}^{d-1})$ is given by a closed oriented $(d-1)$-manifold $M_{d-1}$ equipped with an (automatically closed) $(d-1)$-form $\omega_{d-1;M_{d-1}}$. A morphism $W\colon M^0_{d-1}\to M^1_{d-1}$ in $\Bord_{d,d-1}^\mathrm{or}(\Omega_{cl}^{d-1})$ is the datum of an oriented $d$-manifold $W_d$ with $\partial W_d=M_{d-1}^1\coprod (M_{d-1}^0)^{\mathrm{opp}}$, where ``opp'' denotes the opposite orientation, equipped with a closed $(d-1)$-form $\omega_{d-1;W_d}$ such that
\[
\omega_{d-1;W_d}\bigr\vert_{M_{d-1}^i}=\omega_{d-1;M_{d-1}^i}
\]
for $i=0,1$.
\end{example}

\begin{definition}\label{def:TQFT}
Let $\mathcal{C}$ be a symmetric monoidal category. A $(d,d-1)$-dimensional $\mathcal{C}$-valued topological quantum field theory (TQFT for short) with tangential structure $\xi$ and target $X$ is a
symmetric monoidal functor 
\[
Z\colon \Bord_{d,d-1}^\xi(X)\to \mathcal{C}.
\]
\end{definition}
A typical target is $\mathcal{C}=\mathsf{Vect}$, the category of vector spaces (over some fixed field $\mathbb{K}$). In this case the functor $Z$ maps the $(d-1)$-manifold $M_{d-1}$ (with tangential structure and background fields) to a vector space $V_{M_{d-1}}$ and the $d$-manifold $W_d$ to a liner map $\varphi_W\colon V_{M_{d-1}^0}\to V_{M_{d-1}^1}$. Equivalently, if $W_d$ is a $d$-manifold with $\partial W_d=M_{d-1}$, then $Z(W_d)$ is an element of $V_{M_{d-1}}$.\footnote{This corresponds to thinking of all of the boundary of $W_d$ as ``outgoing'' so that $M_{d-1}^0=\emptyset$, and to identifying elements of $V_{M_{d-1}}$ with linear maps $\mathbb{K}\to V_{M_{d-1}}$ by means of the distinguished basis $\{1\}$ of $\mathbb{K}$.} This example is so typical\footnote{It is the original definition of Atiyah \cite{atiyahtqft}} that when one speaks of a TQFT without specifying the target $\mathcal{C}$ one means $\mathcal{C}=\mathsf{Vect}$. Yet there are plenty of interesting targets other than $\mathsf{Vect}$. Here we will be concerned with the symmetric monoidal categories naturally associated with abelian groups and with morphisms of abelian groups.

{\color{black}In the rest of this section, we introduce several examples and constructions that we will assemble in Section \ref{sec:bnr} to obtain our main result.}
\begin{definition}\label{def:A-tens}
Let $(A,+)$ be an abelian group.\footnote{We will be also using the multiplicative notation $(A,\cdot)$.} By $A^\otimes$ we will denote the symmetric monoidal category with
\begin{align*}
\mathsf{Ob}(A^\otimes)&=A;  \\
\\
\mathsf{Hom}_{A^\otimes}(a,b)&=\begin{cases}
\mathrm{id}_a\qquad\text{if $a=b$}\\
\emptyset \qquad\text{otherwise.}
\end{cases}
\end{align*}
The tensor product is given by the sum (or multiplication) in $A$ and the unit object is the zero (or the unit) of $A$. Associators, unitors and braidings are the trivial ones. 
\end{definition}
 \begin{remark}
Notice that $A^\otimes$ is a rigid monoidal category: the dual of an object $a$ is given by the opposite element $-a$ (or the inverse element $a^{-1})$.
\end{remark}
Spelling out Definition \ref{def:TQFT} for $\mathcal{C}=A^\otimes$ we see that a TQFT with tangential structure $\xi$ and target $X$ with values in $A^\otimes$ consists into a rule that associates with any closed $(d-1)$-manifold $M_{d-1}$ (with tangential structure and background fields) an element $Z(M_{d-1})\in A$ in such a way that:
\begin{itemize}
\item $Z(M_{d-1}\sqcup M'_{d-1})=Z(M_{d-1})+Z(M'_{d-1})$ (monoidality);
\item if $M_{d-1}=\partial W_d$ then $Z( M_{d-1})=0$ (functoriality).
\end{itemize}

\begin{example}[Stokes' theorem for closed forms]\label{ex:Stokes-closed}
A paradigmatic example of a TQFT with values in an abelian group is provided by Stokes' theorem. Take the stack $X$ of background fields to be the smooth stack $\Omega_{cl}^{d-1}$ of closed $(d-1)$-forms as in Example \ref{ex:closed-forms}, and let $\mathbb{R}^\otimes$ be the symmetric monoidal category associated with the abelian group $(\mathbb{R},+)$. Then
\begin{align*}
Z_{\mathrm{Stokes}^{\mathrm{cl}}}\colon \Bord_{d,d-1}^{\mathrm{or}}(\Omega_{cl}^{d-1})&\to \mathbb{R}^\otimes\\
(M_{d-1},\omega_{d-1;M_{d-1}})&\mapsto \int_{M_{d-1}}\omega_{d-1;M_{d-1}}
\end{align*}
is a TQFT. The monoidality \[
Z_{\mathrm{Stokes}^{\mathrm{cl}}}((M_{d-1},\omega_{d-1;M_{d-1}})\sqcup(M'_{d-1},\omega_{d-1;M'_{d-1}}))=\int_{M_{d-1}}\omega_{d-1;M_{d-1}}+\int_{M'_{d-1}}\omega_{d-1;M'_{d-1}}
\]is given by the additivity of the integral, and functoriality is precisely Stokes' theorem: if $M_{d-1}=\partial W_d$ and $\omega_{d-1;M_{d-1}}$ is the restriction to $M_{d-1}$ of a closed $(d-1)$-form $\omega_{d-1;W_d}$ on $W_d$, we have 
\[
Z_{\mathrm{Stokes}^{\mathrm{cl}}}(M_{d-1},\omega_{d-1;M_{d-1}})=\int_{M_{d-1}}\omega_{d-1;M_{d-1}}=\int_{\partial W_{d}}\omega_{d-1;W_d}=\int_{W_{d}}\mathrm{d}\omega_{d-1;W_d}=0.
\]

\end{example}

\begin{example}[$\mathbb{R}$-valued oriented bordism invariants]\label{rem:chern-weil}
Chern-Weil theory provides differential form representatives for Pontryagin classes. These representatives can be seen as morphisms of smooth stacks
\[
\mathbf{p}^{\mathrm{CW}}_k\colon \mathbf{B}\mathrm{SO}_\nabla \to \Omega_{cl}^{4k},
\]
where $\mathbf{B}\mathrm{SO}_\nabla=\varinjlim \mathbf{B}\mathrm{SO}(n)_\nabla$ is the stack of principal $\mathrm{SO}$-bundles with connections, {\color{black} see, e.g.,  \cite{FreedHopkins} and \cite{FSS})}. We have an induced symmetric monoidal morphism
\[
\Bord_{4k+1,4k}^{\mathrm{or}}(\mathbf{B}\mathrm{SO}_\nabla)\to \Bord_{4k+1,4k}^{\mathrm{or}}(\Omega_{cl}^{4k})
\]
and so a TQFT
\begin{align*}
Z_{\mathrm{CW}}\colon \Bord_{4k+1,4k}^{\mathrm{or}}(\mathbf{B}\mathrm{SO}_\nabla)&\to \mathbb{R}^\otimes\\
(M_{4k},P_{M_{4k}}, \nabla_{P_{M_{4k}}})&\mapsto \int_{M_{4k}}\mathbf{p}^{\mathrm{CW}}_k(\nabla_{P_{M_{4k}}}).
\end{align*}
If one keeps the principal bundle $P_{M_{4k}}$ fixed and changes the $\mathfrak{so}$-connection $\nabla_{P_{M_{4k}}}$ into a new connection $\nabla_{P_{M_{4k}}}'$, the difference $\mathbf{p}^{\mathrm{CW}}_k(\nabla_{P_{M_{4k}}})-\mathbf{p}^{\mathrm{CW}}_k(\nabla_{P_{M_{4k}}}')$ is an exact form, so that our TQFT descends to a TQFT with target fields $\mathbf{B}\mathrm{SO}$, i.e., we have a commutative diagram
\[
\begin{tikzcd}
\Bord_{4k+1,4k}^{\mathrm{or}}(\mathbf{B}\mathrm{SO}_\nabla)\ar[rr]\ar[d] &&\mathbb{R}^\otimes\\
\Bord_{4k+1,4k}^{\mathrm{or}}(\mathbf{B}\mathrm{SO})\ar[rru]&&{\quad.}
\end{tikzcd}
\]
The ($(4k+1)$-stabilized) tangent bundle provides a symmetric monoidal section to the forgetful morphism $\Bord_{4k+1,4k}^{\mathrm{or}}(\mathbf{B}\mathrm{SO})\to \Bord_{4k+1,4k}^{\mathrm{or}}$, so we get an absolute (i.e., with trivial background) oriented TQFT
\begin{align*}
Z_{\mathrm{PT}}\colon \Bord_{4k+1,4k}^{\mathrm{or}}&\to \mathbb{R}^\otimes\\
M_{4k}&\mapsto \int_{M_{4k}}\mathbf{p}_k(TM_{4k}).
\end{align*}
The same argument applies replacing the single Pontryagin class $p_{k}$ with a polynomial $\Phi=\Phi(p_1,p_2,\dots)$ in the Pontryagin classes. This way one obtains plenty of $\mathbb{R}$-valued oriented TQFTs. These are in particular $\mathbb{R}$-valued oriented bordism invariants, and Thom's isomorphism
\[
\mathbf{\Omega}^{\mathrm{SO}}_\bullet\otimes \mathbb{R}\cong \mathbb{R}[p_1,p_2,\dots]
\]
implies that indeed every $\mathbb{R}$-valued oriented bordism invariant is of this form.
\end{example}
{\color{black}
\begin{remark}
 It is an interesting question whether the formalism of functorial field theories is flexible enough to capture also the torsion part of bordism invariants. The results in this note, as well as the closely related ones by Berwick-Evans \cite{b-e:How-do-field-theories-detect-the-torsion} show that at least in some cases this is indeed so, and possibly suggest one could expect an affirmative answer to this question in general. In the concluding Remark \ref{rem:further-examples} we discuss the possibility of realizing the isomorphism 
 \[
 \Bord_7^{\mathrm{Fivebrane}}\cong \Bord_7^{\mathrm{fr}}\cong \Z/240\Z
 \]
 along the same lines presented here for the isomorphism $\Bord_3^{\mathrm{String}}\cong \Bord_3^{\mathrm{fr}}\cong \Z/24\Z$.
\end{remark}
}

More generally, one can associate a symmetric monoidal category with a morphism of abelian groups, as follows.   
\begin{definition}\label{def:phi-tens}
Let $\varphi_A\colon A_{\mathrm{mor}}\to A_{\mathrm{ob}}$ be a morphism of abelian groups. By $\varphi_A^\otimes$ we will denote the symmetric monoidal category with
\begin{align*}
\mathsf{Ob}(\varphi_A^\otimes)&=A_{\mathrm{ob}};  \\
\\
\mathsf{Hom}_{\varphi_A^\otimes}(a,b)&=\{x\in A_{\mathrm{mor}}\,:\, a+\varphi_A(x)=b\}.
\end{align*}
The composition of morphism is given by the sum in $A_{\mathrm{mor}}$.
The tensor product of objects and morphisms is given by the sum  in $A_{\mathrm{ob}}$ and in $A_{\mathrm{mor}}$, respectively. The unit object is the zero in $A_{\mathrm{ob}}$. Associators, unitors and braidings are the trivial ones, i.e., they are given by the zero in $A_{\mathrm{mor}}$. 
\end{definition}
\begin{remark}
It is easy to see that Definition \ref{def:phi-tens} is indeed a generalization of Definition \ref{def:A-tens}: if $\iota_A\colon \mathbf{0}\to A$ denotes the initial morphism for $A$, then one has an evident isomorphism $\iota_A^\otimes\cong A^\otimes$.   
\end{remark}
As we did for TQFTs with values in $A^\otimes$, we can spell out the data of  a TQFT (with tangential structure $\xi$ and target $X$) with values in $\varphi_A^\otimes$. It consists of a rule that associates with any closed $(d-1)$-manifold $M_{d-1}$ (with tangential structure and background fields) an element $Z(M_{d-1})\in A_{\mathrm{ob}}$, and with any $d$-manifold $W_{d}$ (with tangential structure and background fields) an element $Z(W_d)\in A_{\mathrm{mor}}$ in such a way that:
\begin{itemize}
\item $Z(M_{d-1}\sqcup M'_{d-1})=Z(M_{d-1})+Z(M'_{d-1})$ and $Z(W_d\sqcup W'_d)=Z(W_d)+Z(W'_d)$ (monoidality);
\item if $M_{d-1}=\partial W_d$ then $Z( M_{d-1})=\varphi_A\left(Z(W_d)\right)$ (functoriality).
\end{itemize} 
The following example is the immediate generalization of Example \ref{ex:Stokes-closed}.
\begin{example}[Stokes' theorem]\label{ex:Stokes}
 Take as stack of background fields the smooth stack $\Omega^{d-1}$ of smooth $(d-1)$-forms. Then we have a TQFT
 \begin{align*}
 Z_{\mathrm{Stokes}}\colon \Bord_{d,d-1}^{\mathrm{or}}(\Omega^{d-1})&\to \mathrm{id}_\mathbb{R}^\otimes\\
(M_{d-1},\omega_{d-1;M_{d-1}})&\mapsto \int_{M_{d-1}}\omega_{d-1;M_{d-1}}\\
(W_d,\omega_{d-1;W_d})&\mapsto \int_{W_d}\mathrm{d}\omega_{d-1;W_d}.
\end{align*}
As in Example \ref{ex:Stokes-closed}, monoidality is given by the additivity of the integral and functoriality is given by Stokes' theorem: if $M_{d-1}=\partial W_d$ and $\omega_{d-1;M_{d-1}}$ is the restriction to $M_{d-1}$ of a $(d-1)$-form $\omega_{d-1;W_d}$ on $W_d$, then we have 
\begin{align*}
Z_{\mathrm{Stokes}}(M_{d-1},\omega_{d-1;M_{d-1}})&=\int_{M_{d-1}}\omega_{d-1;M_{d-1}}=\int_{\partial W_{d}}\omega_{d-1;W_d}=\int_{W_{d}}\mathrm{d}\omega_{d-1;W_d}\\
&=Z_{\mathrm{Stokes}}(W_d,\omega_{d-1;W_d})=\mathrm{id}_\mathbb{R}\left(Z_{\mathrm{Stokes}}(W_d,\omega_{d-1;W_d})\right).
\end{align*}

\end{example}

\begin{example}[Holonomy and curvature]\label{ex:hol-curv}
    A generalization of the above Example for $d=2$ is obtained by taking $X=\mathbf{B}\mathrm{U}(1)_{\nabla}$, the stack of principal $\mathrm{U}(1)$-bundles with connection. As target we take the symmetric monoidal category  associated with the exponential morphism
    \[
    \exp(2\pi i -)\colon \mathbb{R}\to \mathrm{U}(1).
    \]
    There's a natural TQFT produced by these data and given by
    \begin{align*}
 Z_{\mathrm{hol}}\colon \Bord_{2,1}^{\mathrm{or}}(\mathbf{B}\mathrm{U}(1)_{\nabla})&\to \exp(2\pi i -)^\otimes\\
(M_{1},P_{M_1},\nabla^{}_{{M_1}})&\mapsto \mathrm{hol}_{M_1}(\nabla^{}_{{M_1}})\\
(W_2,P_{W_2},\nabla^{}_{{W_2}})&\mapsto \frac{1}{2\pi i}\int_{W_2}F_{\nabla^{}_{{W_2}}},
\end{align*}
where $\mathrm{hol}_{M_1}(\nabla)$ is the holonomy of the connection $\nabla$ along the closed oriented 1-manifold $M_1$ (i.e., the usual holonomy along a copy of $S^1$, extended monoidally using the fact that $M_1$ is a disjoint union of finitely many copies of $S^1$), and $F_\nabla$ is the curvature 2-form of $\nabla$. The fact that $Z_{\mathrm{hol}}$ is a TQFT is encoded in the fundamental integral identity relating holonomy along the boundary and curvature in the interior: 
\[
\mathrm{hol}_{\partial W_2}(\nabla_{W_2})=\exp\left( \int_{W_2}F_{\nabla_{W_2}}\right),
\]
{\color{black}see, e.g., \cite{konrad-qualcosa}}.
\end{example}
\begin{remark}
The stack $\mathbf{B}\mathrm{U}(1)_{\nabla}$ can be seen as the stack associated by the Dold-Kan correspondence with the chain complex of sheaves  
\[
C^\infty(-;\mathrm{U}(1))\xrightarrow{\frac{1}{2\pi i}\mathrm{d}\mathrm{log}} \Omega^1(-),
\]
with $C^\infty(-;\mathrm{U}(1))$ in degree 1. This description precisely encodes the fact that the local data for a principal $\mathrm{U}(1)$-bundle with connection $(P,\nabla)$  over a smooth manifold $M$ endowed with an open cover $\mathcal{U}=\{U_\alpha\}_\alpha\in I$ are given by
\begin{itemize}
    \item smooth maps $g_{\alpha\beta}\colon U_{\alpha\beta}\to \mathrm{U}(1)$
    \item 1-forms $A_\alpha$ on $U_\alpha$
\end{itemize}
such that
\begin{itemize}
    \item $g_{\alpha\beta} g_{\beta\gamma} g_{\gamma\alpha}=1$ on $U_{\alpha\beta\gamma}$;
    \item $A_\beta-A_\alpha=\frac{1}{2\pi i}\mathrm{d}\mathrm{log}(g_{\alpha\beta})$ on $U_{\alpha\beta}$.
\end{itemize}
If $M$ is a smooth manifold, the set of homotopy classes of maps from $M$ to $\mathbf{B}\mathrm{U}(1)_{\nabla}$ classifies the isomorphism classes of principal $\mathrm{U}(1)$-bundles with connection over $M$, so that
\[
[M,\conn{\B \mathrm{U}(1)}]\cong \hat{H}^2(M;\mathbb{Z}),
\]
where $\hat{H}^n(-;\mathbb{Z})$ denotes degree $n$ ordinary differential cohomology (or Deligne cohomology), see, e.g., \cite{bryl,FSS}. In terms of the Dold-Kan correspondence, this reduces to the usual definition of ordinary differential cohomology as an hypercohomology:
\[
\hat{H}^n(M;\mathbb{Z})\cong \mathbb{H}^n(M;C^\infty(-;\mathrm{U}(1))\to \Omega^1)
\cong \mathbb{H}^n(M;\mathbb{Z}\to\Omega^0\to \Omega^1).
\]
\end{remark}
{\color{black}
\begin{remark}
  By forgetting the connection, one has a forgetful morphism $\mathbf{B}\mathrm{U}(1)_{\nabla}\to \mathbf{B}\mathrm{U}(1)$. It is a nice exercise, possibly  helpful in getting familiar with homotopy pullbacks of smooth stacks, to show that we have a homotopy pullback   
  \[
\begin{tikzcd}
	{\Omega^1} & {\conn{\mathbf{B}\mathrm{U}(1)}} \\
	{\ast} & {\mathbf{B}\mathrm{U}(1),}
	\arrow[from=1-1, to=2-1]
	\arrow[from=2-1, to=2-2]
	\arrow[from=1-2, to=2-2]
	\arrow[from=1-1, to=1-2]
		\arrow[phantom, "\usebox\pullback" , very near start, color=black,from=1-1, to=2-2]
\end{tikzcd}
\]
where $\ast\to \mathbf{B}\mathrm{U}(1)$ is the map picking the trivial $U(1)$-bundle. In terms of local data $(g_{\alpha\beta},A_\alpha)$ for a principal $U(1)$-bundle with connection on a manifold $M$, this corresponds to the fact that for any choice of a trivialization $(h_\alpha)$ of the cocycle $(g_{\alpha\beta})$, i.e., for any choice of smooth functions $h_\alpha\colon U_\alpha\to \mathrm{U}(1)$ such that $g_{\alpha\beta}=h_{\beta}{h_\alpha}^{-1}$, the local 1-forms $A_\alpha-\frac{1}{2\pi i}\mathrm{d}\mathrm{log}(h_{\alpha})$ define a global 1-form on $M$.
\end{remark}
}
\begin{remark}
A second, equivalent, description of the stack $\mathbf{B}\mathrm{U}(1)_{\nabla}$ is as the homotopy pullback
\[
\begin{tikzcd}
	{\conn{\mathbf{B}\mathrm{U}(1)}} & {\Omega^{2}_{cl}} \\
	{\mathbf{B}^{2}\Z} & {\mathbf{B}^{2}\R,}
	\arrow[from=1-1, to=2-1]
	\arrow[from=2-1, to=2-2]
	\arrow[from=1-2, to=2-2]
	\arrow[from=1-1, to=1-2,"\frac{1}{2\pi i}F"]
		\arrow[phantom, "\usebox\pullback" , very near start, color=black,from=1-1, to=2-2]
\end{tikzcd}
\]
where $F$ is the map sending  $\mathrm{U}(1)$-connection to its curvature and $\Omega^2_{cl}\to \mathbf{B}\R$ is actually the span
\[
\Omega^2_{cl}\to \left(\Omega^0\to\Omega^1\to \Omega^2_{cl}\right)\xleftarrow{\sim}\mathbf{B}\R, 
\]
and is the map inducing at the cohomology level the morphism $\Omega^2_{cl}(M)\to H^2_{\mathrm{dR}}(M)\cong H^2(M;\mathbb{R})$ mapping a closed 2-form to its de Rham cohomology class. This description is the one that directly encodes the naive idea of an ordinary differential cohomology class as the datum of integral cohomology class together with a closed differential form representing it, see, e.g., \cite[Proposition 3.24]{qualcosadibunke}. 
\end{remark}

\begin{example}[Higher holonomies and curvatures]\label{ex:high-hol-curv}
Example \ref{ex:hol-curv} admits an immediate higher dimensional generalization. One considers the higher stack $\conn{\mathbf{B}^n\mathrm{U}(1)}$ of $\mathrm{U}(1)$-$n$-principal bundles with connection, that can be equivalently defined 
as the homotopy pullback 
\[
\begin{tikzcd}
	{\conn{\mathbf{B}^n\mathrm{U}(1)}} & {\Omega^{n+1}_{cl}} \\
	{\mathbf{B}^{n+1}\Z} & {\mathbf{B}^{n+1}\R}
	\arrow[from=1-1, to=2-1]
	\arrow[from=2-1, to=2-2]
	\arrow[from=1-2, to=2-2]
	\arrow[from=1-1, to=1-2,"\frac{1}{2\pi i}F"]
		\arrow[phantom, "\usebox\pullback" , very near start, color=black,from=1-1, to=2-2]
\end{tikzcd}
\]
or via the Dold-Kan correspondence, as the stack associated by the $n$-terms Deligne complex
\[
C^\infty(-,\mathrm{U}(1))\xrightarrow{\frac{1}{2i\pi}\mathrm{d}\mathrm{log}}\Omega^1\xrightarrow{\mathrm{d}}\Omega^2\to\cdots\xrightarrow{\mathrm{d}}\Omega^n,
\]
and defines an $(n+1,n)$-dimensional TQFT as
   \begin{align*}
 Z_{\mathrm{hol}}^{(n)}\colon \Bord_{n+1,n}^{\mathrm{or}}(\conn{\mathbf{B}^n\mathrm{U}(1)})&\to \exp(2\pi i -)^\otimes\\
(M_{n},P_{M_n},\nabla_{M_n})&\mapsto \mathrm{hol}_{M_n}(\nabla^{}_{M_n})\\
(W_{n+1},P_{W_{n+1}},\nabla_{W_{n+1}})&\mapsto \frac{1}{2\pi i}\int_{W_{n+1}}F_{\nabla^{}_{W_{n+1}}}.
\end{align*}
\begin{remark}\label{rem:curvature-and-deRham}
The morphism of chain complexes
\[
\begin{tikzcd}
	0 & 0 & \cdots & 0 & \Omega^n \\
	C^\infty(-,\mathrm{U}(1)) & \Omega^1 &\cdots & \Omega^{n-1} & \Omega^n
	\arrow[from=1-1, to=1-2]
	\arrow[from=1-2, to=1-3]
	\arrow[from=1-3, to=1-4]
	\arrow[from=1-4, to=1-5]
	\arrow[from=2-1, to=2-2,"\frac{1}{2\pi i}\mathrm{d}\mathrm{log}"]
    \arrow[from=2-2, to=2-3,"\mathrm{d}"]
	\arrow[from=2-3, to=2-4]
	\arrow[from=2-4, to=2-5,"\mathrm{d}"]
	\arrow[from=1-1, to=2-1]
	\arrow[from=1-2, to=2-2]
	\arrow[from=1-3, to=2-3]
	\arrow[from=1-4, to=2-4]
	\arrow[from=1-5, to=2-5,"\mathrm{id}"]
\end{tikzcd}
\]
induces a morphism of smooth stacks $\Omega^n\to\conn{\mathbf{B}^n\mathrm{U}(1)}$ interpreting $n$-forms as particular connections on trivial $\mathrm{U}(1)$-$n$-bundles. For these particular connections, holonomy along a closed oriented $n$-manifold reduces to integration. More precisely, we have a commutative diagram
\[
\begin{tikzcd}
	{[M_n,\Omega^n]} & {[M_n,\conn{\mathbf{B}^n\mathrm{U}(1)}]} \\
	\mathbb{R} & {\mathrm{U}(1).}
	\arrow[from=1-1, to=2-1,"\int"']
	\arrow[from=1-2, to=2-2,"\mathrm{hol}"]
	\arrow[from=2-1, to=2-2,"\exp(2\pi i-)"]
	\arrow[from=1-1, to=1-2]
\end{tikzcd}
\]
Notice how the additivity of the integral is translated into the multiplicativity of holonomy by the group homomorphism $\exp(2\pi i-)$. Moreover, for $n$-forms seen as particular connections, the curvature is identified with the de Rham differential; namely, we have  a
commutative diagram
\[
\begin{tikzcd}
	{\Omega^n} & {\Omega^{n+1}_{cl}} \\
	{\conn{\mathbf{B}^n\mathrm{U}(1)}}. & 
	\arrow[from=1-1, to=2-1]
	\arrow[from=2-1, to=1-2,"\frac{1}{2\pi i}F"']
	\arrow[from=1-1, to=1-2,"\mathrm{d}"]
\end{tikzcd}
\]
\end{remark}
\end{example}
With our last Example in this Section we finally connect to spin structures.
\begin{example}[TQFTs from spin connections]\label{ex:tqft-spin-conn}
Let $\mathrm{BSpin}=\varinjlim \mathrm{BSpin}(n)$ be the classifying space for the infinite spin group. One has $H^4(\mathrm{BSpin};\Z)\cong \Z$, with a generator given by the first fractional Pontryagin class $\frac{1}{2}p_1$. This can equivalently be seen as a map (well defined up to homotopy) $\frac{1}{2}p_1:\mathrm{BSpin}\to K(\Z,4)$. Also the Eilenberg-MacLane space $K(\Z,4)$ can be seen as a classifying space: one has $K(\Z,4)\cong B^3\mathrm{U}(1)$, i.e.,  $K(\Z,4)$ is the classifying space for $\mathrm{U}(1)$-principal 3-bundles. Brylinksi and McLaughlin  show in \cite{BML} how one can construct  a \v{C}ech cocycle representative for $\frac{1}{2}p_1$ starting from a \v{C}ech cocycle representative for a principal Spin bundle. This construction can be reinterpreted as a refinement of the map of classifying spaces  $\frac{1}{2}p_1:\mathrm{BSpin}\to B^3\mathrm{U}(1)$ to a morphism of smooth stacks
\[
\frac{1}{2}\mathbf{p}_1:\mathbf{B}\mathrm{Spin}\to \mathbf{B}^3\mathrm{U}(1).
\]
Brylinski and McLaughlin also show how from the cocycle data for a spin connection on a principal Spin bundle one can construct a Deligne cocycle representative for  an ordinary differential cohomology class $\frac{1}{2}\hat{p}_1$ lifting $\frac{1}{2}p_1$ to  differential cohomology. This too admits a  natural interpretation in terms of smooth stacks: it is a lifting of $\frac{1}{2}\mathbf{p}_1$ to a morphism of smooth stacks $\frac{1}{2}\hat{\mathbf{p}}_1:\conn{\mathbf{B}\mathrm{Spin}}\to \conn{\mathbf{B}^3\mathrm{U}(1)}$, i.e., we have a
 commutative diagram
 \begin{equation}\label{diag:spin-conn}
\begin{tikzcd}
    \conn{\mathbf{B}\mathrm{Spin}}\arrow[r,"\frac{1}{2}\hat{\mathbf{p}_1}"]\arrow[d]&\conn{\mathbf{B}^3\mathrm{U}(1)}\arrow[d]\\
    \mathbf{B}\mathrm{Spin}\arrow[r,"\frac{1}{2}\mathbf{p}_1"]&\mathbf{B}^3\mathrm{U}(1),
\end{tikzcd}
 \end{equation}
Composing with the curvature morphism $\frac{1}{2\pi i}F\colon \conn{\mathbf{B}^3U(1)}\to\Omega^4_{cl}$ from Example \ref{ex:high-hol-curv}, we obtain the commutative diagram
\begin{equation}\label{diag:CW}
\begin{tikzcd}
\conn{\mathbf{B}\mathrm{Spin}} \arrow[r, "\frac{1}{2}\hat{\mathbf{p}}_1"'] \arrow[d] \arrow[rr, "\frac{1}{2}\mathbf{p}_1^{\mathrm{CW}}", bend left=15] & \conn{\B^3\mathrm{U}(1)} \arrow[d] \arrow[r, "\frac{1}{2\pi i}F"'] & \Omega_{cl}^4 \\
\mathbf{B}\mathrm{Spin} \arrow[r, "\frac{1}{2}\mathbf{p}_1"']                                                                      & {\mathbf{B}^3\mathrm{U}(1),}                               &              
\end{tikzcd}    
\end{equation}
where the bent arrow on the top is the Chern-Weil representative for $\frac{1}{2}p_1$ seen as a morphism of smooth stacks (see Remark \ref{rem:chern-weil}). Therefore, we obtain a $(4,3)$-dimensional TQFT with background fields given by spin connections and target $\exp(2\pi i -)^\otimes$ by setting
  \begin{align*}
 Z_{\mathrm{Spin}}\colon \Bord_{4,3}^{\mathrm{or}}(\conn{\mathbf{B}\mathrm{Spin}})&\to \exp(2\pi i -)^\otimes\\
(M_{3},P_{M_3},\nabla^{}_{M_3})&\mapsto \mathrm{hol}_{M_3}\left(\frac{1}{2}\hat{\mathbf{p}}_1(\nabla^{}_{M_3})\right)\\
(W_{4},P_{W_4},\nabla^{}_{W_4})&\mapsto \int_{W_{4}}\frac{1}{2}\mathbf{p}_1^{\mathrm{CW}}(\nabla^{}_{W_4}).
\end{align*}
\end{example}

\section{Geometric string structures}
In this Section we recall Waldorf's notion of a geometric string structure \cite{waldorf}. We present it by using the language of smooth stacks. Doing this, on the one hand the presentation is  extremely fast, and on the other hand all the main features of geometric string structures are very easily derived.

To begin with, recall that the classifying space of the string group is given, essentially by definition, by the homotopy fiber of the first fractional Pontryagin class. In other words $\mathrm{BString}$ is defined as the homotopy fiber product of spaces
\[
\begin{tikzcd}
	{\mathrm{BString}} & \ast \\
	{\mathrm{BSpin}} & {K(\mathbb{Z},4).}
	\arrow[from=1-1, to=2-1]
	\arrow[from=2-1, to=2-2,"\frac{1}{2}p_1"]
	\arrow[from=1-2, to=2-2]
	\arrow[from=1-1, to=1-2]
		\arrow[phantom, "\usebox\pullback" , very near start, color=black,from=1-1, to=2-2]
\end{tikzcd}
\]
We have already remarked that $K(\mathbb{Z},4)\cong B^3\mathrm{U}(1)$, so that $\mathrm{BString}$ is equivalently defined as the homotopy fiber of $\frac{1}{2}p_1\colon  \mathrm{BSpin}\to B^3\mathrm{U}(1)$. We have also recalled that, seen this way, the first fractional Pontryagin class can be refined to a morphism of smooth stacks $\frac{1}{2}\mathbf{p}_1:\mathbf{B}\mathrm{Spin}\to \mathbf{B}^3\mathrm{U}(1)$. This naturally leads to defining the smooth stack of principal string bundles as the homotopy pullback of smooth stacks
\[
\begin{tikzcd}
	{\B \mathrm{String}} & \ast \\
	{\B \mathrm{Spin}} & {\B^3\mathrm{U}(1)}.
	\arrow[from=1-1, to=2-1]
	\arrow[from=2-1, to=2-2,"\frac{1}{2}\mathbf{p}_1"]
	\arrow[from=1-2, to=2-2]
	\arrow[from=1-1, to=1-2]
		\arrow[phantom, "\usebox\pullback" , very near start, color=black,from=1-1, to=2-2]
\end{tikzcd}
\]
Explicitly, this means that a principal string bundle over a manifold $M$ is the datum of  principal spin bundle $P$ over $M$ \emph{together} with the choice of a trivialization of the associated principal $\mathrm{U}(1)$-3-bundle. Notice that this contains both a topological condition (the $\mathrm{U}(1)$-3-bundle has to be trivializable, and this is equivalent to $\frac{1}{2}p_1(P)=0\in H^4(M;\mathbb{Z})$) and additional structure (the choice of a trivialization).

We also noticed that the morphism of stacks $\frac{1}{2}\mathbf{p}_1:\mathbf{B}\mathrm{Spin}\to \mathbf{B}^3\mathrm{U}(1)$ admits a refinement to a morphism of stacks $\frac{1}{2}\hat{\mathbf{p}}_1:\conn{\mathbf{B}\mathrm{Spin}}\to \conn{\mathbf{B}^3\mathrm{U}(1)}$, and this leads to defining the stack of principal string connections as the homotopy pullback 
\[
\begin{tikzcd}
	{\conn{\B \mathrm{String}}} & \ast \\
	{\conn{\B \mathrm{Spin}}} & {\conn{\B^3\mathrm{U}(1)}}
	\arrow[from=1-1, to=2-1]
	\arrow[from=2-1, to=2-2,"\frac{1}{2}\hat{\mathbf{p}}_1"]
	\arrow[from=1-2, to=2-2]
	\arrow[from=1-1, to=1-2]
		\arrow[phantom, "\usebox\pullback" , very near start, color=black,from=1-1, to=2-2]
\end{tikzcd}.
\]
Yet, it is not $\conn{\B \mathrm{String}}$ the stack we are interested here, but a variant of it that has a ``more topological'' nature as we are going to explain in Proposition \ref{prop:geom-string-structures-exist} and in Remark \ref{rem:do-string-connections-exist?}. To begin with, notice that the sequence of forgetful morphisms of chain complexes 
\[
\begin{tikzcd}
{C^\infty(-;\mathrm{U}(1))} & {\Omega^1} & {\Omega^2} & {\Omega^3}\\
{C^\infty(-;\mathrm{U}(1))} & {\Omega^1} & {\Omega^2} & 0\\
{C^\infty(-;\mathrm{U}(1))} & {\Omega^1} & 0 & 0\\
{C^\infty(-;\mathrm{U}(1))} & 0 & 0 & 0
	\arrow[from=1-1, to=1-2, ,"\frac{1}{2\pi i}\mathrm{d}\mathrm{log}"]
    \arrow[from=1-2, to=1-3,"\mathrm{d}"]
    \arrow[from=1-3, to=1-4,"\mathrm{d}"]
    \arrow[from=2-1, to=2-2,"\frac{1}{2\pi i}\mathrm{d}\mathrm{log}"]
    \arrow[from=2-2, to=2-3,"\mathrm{d}"]
    \arrow[from=2-3, to=2-4]
    \arrow[from=3-1, to=3-2,"\frac{1}{2\pi i}\mathrm{d}\mathrm{log}"]
    \arrow[from=3-2, to=3-3]
    \arrow[from=3-3, to=3-4]
    \arrow[from=4-1, to=4-2]
    \arrow[from=4-2, to=4-3]
    \arrow[from=4-3, to=4-4]
    \arrow[equal,from=1-1, to=2-1]
    \arrow[equal,from=1-2, to=2-2]
    \arrow[equal,from=1-3, to=2-3]
    \arrow[from=1-4, to=2-4]
    \arrow[equal,from=2-1, to=3-1]
    \arrow[equal,from=2-2, to=3-2]
    \arrow[from=2-3, to=3-3]
    \arrow[equal,from=2-4, to=3-4]
    \arrow[equal,from=3-1, to=4-1]
    \arrow[from=3-2, to=4-2]
    \arrow[equal,from=3-3, to=4-3]
    \arrow[equal,from=3-4, to=4-4]
\end{tikzcd}
\]
induces the sequence of morphisms of smooth stacks
\[
\conn{\B^3\mathrm{U}(1)}\to \B\conn{\B^2\mathrm{U}(1)} \to \B^2\conn{\B \mathrm{U}(1)}\to \B^3\mathrm{U}(1),
\]
{\color{black}where we have written $\B^i\conn{\B^j\mathrm{U}(1)}$ for  $\B^i(\conn{\B^j\mathrm{U}(1)})$.
}
Precomposing this on the left with $\frac{1}{2}\hat{\mathbf{p}}_1:\conn{\mathbf{B}\mathrm{Spin}}\to \conn{\mathbf{B}^3\mathrm{U}(1)}$ we obtain maps
\[
\frac{1}{2}\hat{\mathbf{p}}_1^{(i)}\colon \conn{\mathbf{B}\mathrm{Spin}}\to \B^{3-i}\conn{\B^{i}\mathrm{U}(1)},
\]
for $i=0,\dots,3$. 
\begin{definition}\label{def:geom-string-struct}
For $i=0,\dots,3$, the smooth stack $\conn{\B \mathrm{String}^{(i)}}$ is defined as the homotopy pullback
\[
\begin{tikzcd}
	{\conn{\B \mathrm{String}^{(i)}}} & \ast \\
	{\conn{\B \mathrm{Spin}}} & {\B^{3-i}\conn{\B^{i}\mathrm{U}(1)}}
	\arrow[from=1-1, to=2-1]
	\arrow[from=2-1, to=2-2,"\frac{1}{2}\hat{\mathbf{p}}_1^{(i)}"]
	\arrow[from=1-2, to=2-2]
	\arrow[from=1-1, to=1-2]
		\arrow[phantom, "\usebox\pullback" , very near start, color=black,from=1-1, to=2-2]
\end{tikzcd}.
\]
The stack $\conn{\B \mathrm{String}^{(2)}}$ will be called the stack of \emph{geometric string structures}.
\end{definition}
\begin{remark}
In terms of Definition \ref{def:geom-string-struct}, the stack $\conn{\B \mathrm{String}}$ of principal string connections is $\conn{\B \mathrm{String}^{(3)}}$ and the morphism $\frac{1}{2}\hat{\mathbf{p}}_1$ is $\frac{1}{2}\hat{\mathbf{p}}_1^{(3)}$. If $M$ is a smooth manifold, then a map $M\to \conn{\B \mathrm{String}^{(3)}}$ consists in the datum of a principal spin connection on $M$ together with a trivialization of the associated $\mathrm{U}(1)$-3-connection. On the other hand, a map $M\to \conn{\B \mathrm{String}^{(0)}}$ consists in the datum of a principal spin connection on $M$ together with a trivialization of the associated $\mathrm{U}(1)$-3-bundle, entirely forgetting the connection data. A geometric string structure on $M$ lies between these two extremes: it consists into a principal spin connection on $M$ together with a trivialization of part of the data of the associated $\mathrm{U}(1)$-3-connection.
\end{remark}
{\color{black}
\begin{remark}
 Twisted versions of the stacks $\conn{\B \mathrm{String}^{(i)}}$, with a topological twist coming from a principal $E_8$-bundle, have been considered in \cite{FSS1,FSS2} in connection with the geometry of $C$-field configurations in M-theory.
\end{remark}
}\begin{remark}\label{rem:loop-space-torsor}
By definition, the datum of a geometric string structure on a manifold $M$ is equivalently the datum of a homotopy commutative diagram
\[
\begin{tikzcd}
	{M} & \ast \\
	{\conn{\B \mathrm{Spin}}} & {\B\conn{\B^{2}\mathrm{U}(1)}}
	\arrow[from=1-1, to=2-1]
	\arrow[from=2-1, to=2-2,"\frac{1}{2}\hat{\mathbf{p}}_1^{(2)}"]
	\arrow[from=1-2, to=2-2]
	\arrow[from=1-1, to=1-2]
\end{tikzcd}.
\]
As a consequence, \emph{two} geometric string structures on $M$ provide the fillers for the homotopy commutative diagram 
\[
\begin{tikzcd}
	& {M} &  \\
	\ast &{\conn{\B \mathrm{Spin}}} & \ast\\
	&{\B\conn{\B^{2}\mathrm{U}(1)}}&
	\arrow[from=1-2, to=2-2]
	\arrow[from=2-2, to=3-2,"\frac{1}{2}\hat{\mathbf{p}}_1^{(2)}"]
	\arrow[from=1-2, to=2-1,bend right=15]
	\arrow[from=1-2, to=2-3,bend left=15]
	\arrow[from=2-1, to=3-2,bend right=15]
	\arrow[from=2-3, to=3-2,bend left=15]
\end{tikzcd}.
\]
This shows that the first geometric string structure is changed into the second by the action given by pasting a homotopy commutative diagram of the form \[
\begin{tikzcd}
	& {M} &  \\
	\ast & & \ast\\
	&{\B\conn{\B^{2}\mathrm{U}(1)}}&
	\arrow[from=1-2, to=2-1,bend right=15]
	\arrow[from=1-2, to=2-3,bend left=15]
	\arrow[from=2-1, to=3-2,bend right=15]
	\arrow[from=2-3, to=3-2,bend left=15]
\end{tikzcd}.
\]
By definition of based loop space of a pointed stack, such a diagram is equivalent to a map from $M$ to the based loop space of $\B\conn{\B^{2}\mathrm{U}(1)}$, i.e., to $\conn{\B^{2}\mathrm{U}(1)}$. i.e., we have a unique factorization (with unique fillers, and where all the uniqueness' are up to unique higher homotopies, etc.) 
\[\begin{tikzcd}
	& M \\
	& {\conn{\B^{2}\mathrm{U}(1)}} \\
	\ast && \ast \\
	& {{\B\conn{\B^{2}\mathrm{U}(1)}}}
	\arrow[from=1-2, to=2-2]
	\arrow[from=2-2, to=3-1]
	\arrow[from=3-1, to=4-2]
	\arrow[from=2-2, to=3-3]
	\arrow[from=3-3, to=4-2]
	\arrow[curve={height=18pt}, from=1-2, to=3-1]
	\arrow[curve={height=-18pt}, from=1-2, to=3-3]
	\arrow["\usebox\pullback"{anchor=center, pos=0.125, rotate=-45}, draw=none, from=2-2, to=4-2]
\end{tikzcd}
\]
This shows that equivalence classes of geometric string structures on $M$ are a torsor for $[M,\conn{\B^2 \mathrm{U}(1)}]\cong \hat{H}^3(M;\mathbb{Z})$.
\end{remark}

\begin{remark}\label{rem:forget-conn}
The sequence of forgetful morphisms $\B^{3-i}\conn{\B^{i}\mathrm{U}(1)}\to \B^{4-i}\conn{\B^{i-1}\mathrm{U}(1)}$ refines the commutative diagram 
 \[
\begin{tikzcd}
    \conn{\mathbf{B}\mathrm{Spin}}\arrow[r,"\frac{1}{2}\hat{\mathbf{p}_1}"]\arrow[d]&\conn{\mathbf{B}^3\mathrm{U}(1)}\arrow[d]\\
    \mathbf{B}\mathrm{Spin}\arrow[r,"\frac{1}{2}\mathbf{p}_1"]&\mathbf{B}^3\mathrm{U}(1)
\end{tikzcd}
 \]
 to a commutative diagram
\[
\begin{tikzcd}[column sep=tiny]
	{\conn{\B \mathrm{Spin}}} & {\conn{\B^3\mathrm{U}(1)}} & {\B\conn{\B^2\mathrm{U}(1)}} & {\B^2\conn{\B \mathrm{U}(1)}} & {\B^3\mathrm{U}(1)} \\
	{\B \mathrm{Spin}} &&&& {\B^3\mathrm{U}(1)}
	\arrow["{\frac{1}{2}\mathbf{p}_1}", from=2-1, to=2-5]
	\arrow["{\frac{1}{2}\hat{\mathbf{p}_1}}", from=1-1, to=1-2]
	\arrow[from=1-2, to=1-3]
	\arrow[from=1-3, to=1-4]
	\arrow[from=1-4, to=1-5]
	\arrow[from=1-1, to=2-1]
    \arrow[equal,from=1-5, to=2-5]
\end{tikzcd}    
\]
\end{remark}
\begin{remark}\label{rem:omega3-and-p1}
By the pasting law for homotopy pullbacks,\footnote{I.e., by the fact that in a homotopy commutative diagram of the form
\[
\begin{tikzcd}[ampersand replacement=\&]
	\bullet \& \bullet \& \bullet \\
\bullet \& \bullet \& \bullet
	\arrow[ from=1-1, to=1-2]
	\arrow[ from=1-2, to=1-3]
	\arrow[ from=2-1, to=2-2]
	\arrow[ from=2-2, to=2-3]
 \arrow[ from=1-1, to=2-1]
	\arrow[ from=1-2, to=2-2]
 \arrow[ from=1-3, to=2-3]
\end{tikzcd}    
\]
where the right square is a pullback, the total rectangle is a pullback if and only if the left square is a pullback.
}
the defining diagram for the stack of geometric string structures can be factored as
\[
\begin{tikzcd}
	{\conn{\B \mathrm{String}^{(2)}}} & \Omega^3 & \ast \\
	{\conn{\B \mathrm{Spin}}} & {\conn{\B^{3}\mathrm{U}(1)}} & {\B\conn{\B^{2}\mathrm{U}(1)}}
	\arrow[from=1-1, to=2-1]
	\arrow[from=2-1, to=2-2,"\frac{1}{2}\hat{\mathbf{p}}_1"]
	\arrow[from=2-2, to=2-3]
	\arrow[from=1-2, to=2-2]
	\arrow[from=1-3, to=2-3]
	\arrow[from=1-1, to=1-2, "\omega_3"]
	\arrow[from=1-2, to=1-3]
		\arrow[phantom, "\usebox\pullback" , very near start, color=black,from=1-1, to=2-2]
		\arrow[phantom, "\usebox\pullback" , very near start, color=black,from=1-2, to=2-3]
\end{tikzcd}.
\]
where both squares are homotopy pullbacks. This in particular shows that a geometric string structure comes equipped with a canonical 3-form. From Remark \ref{rem:curvature-and-deRham} and Example \ref{ex:tqft-spin-conn}
we obtain the commutative diagram
\[
\begin{tikzcd}
	{\conn{\B \mathrm{String}^{(2)}}} & \Omega^3 & \\
	{\conn{\B \mathrm{Spin}}} & {\conn{\B^{3}\mathrm{U}(1)}} & {\Omega_{cl}^4}
	\arrow[from=1-1, to=2-1]
	\arrow[from=2-1, to=2-2,"\frac{1}{2}\hat{\mathbf{p}}_1"]
	\arrow[from=2-2, to=2-3,"\frac{1}{2\pi i}F"]
	\arrow[from=1-2, to=2-2]
	\arrow[from=1-1, to=1-2, "\omega_3"]
		\arrow[phantom, "\usebox\pullback" , very near start, color=black,from=1-1, to=2-2]
    \arrow[from=1-2, to=2-3,"\mathrm{d}",bend left=15]
    \arrow[from=2-1, to=2-3,, "\frac{1}{2}\mathbf{p}_1^{\mathrm{CW}}"', bend right=15]
\end{tikzcd}
\]
showing that, if $\omega_{3,M}$ is the canonical 3-form on a smooth manifold $M$ equipped with a geometric string structure and $\nabla$ is the underlying spin connection, then one has
\[
\mathrm{d}\omega_{3,M}=\frac{1}{2}\mathbf{p}_1^{\mathrm{CW}}(\nabla),
\]
i.e., $\omega_{3,M}$ is a trivialization of the Chern-Weil de Rham representative of the first fractional Pontryagin class of the principal spin bundle $P$ on $M$ coming with the choice of a spin connection on $P$.
\end{remark}

\begin{lemma}\label{lem:forget-string}
The sequence of forgetful morphisms $\B^{3-i}\conn{\B^{i}\mathrm{U}(1)}\to \B^{4-i}\conn{\B^{i-1}\mathrm{U}(1)}$ induces a sequence of forgetful morphisms $\conn{\B \mathrm{String}^{(i)}}\to \conn{\B \mathrm{String}^{(i-1)}}$ for $i=1,\dots,3$.
\end{lemma}
\begin{proof}
By the pasting law for homotopy pullbacks, {\color{black} starting with the bottom right corner, one builds} the following homotopy commutative diagram, where each square is a homotopy pullback, and where the $\mathcal{S}_i$ are stacks we give no particular name since we are not specifically interested in them:
\[
\begin{tikzcd}[column sep=tiny]
	{\conn{\B \mathrm{String}^{(3)}}} & \ast \\
	{\conn{\B \mathrm{String}^{(2)}}} & {\Omega^3} & \ast \\
	{\conn{\B \mathrm{String}^{(1)}}} & {\mathcal{S}_3} & {\B\Omega^2} & \ast \\
	{\conn{\B \mathrm{String}^{(0)}}} & {\mathcal{S}_2} & {\mathcal{S}_1} & {\B^2\Omega^1} & \ast \\
	{\conn{\B \mathrm{Spin}}} & {\conn{\B^3\mathrm{U}(1)}} & {\B\conn{\B^2\mathrm{U}(1)}} & {\B^2\conn{\B \mathrm{U}(1)}} & {\B^3\mathrm{U}(1)} 
	\arrow[from=4-1, to=5-1]
	\arrow[from=3-1, to=4-1]
	\arrow[from=2-1, to=3-1]
	\arrow[from=1-1, to=2-1]
	\arrow[from=1-1, to=1-2]
	\arrow[from=2-1, to=2-2]
	\arrow[from=1-2, to=2-2]
	\arrow[from=2-2, to=2-3]
	\arrow["{\frac{1}{2}\hat{\mathbf{p}_1}}", from=5-1, to=5-2]
	\arrow[from=5-2, to=5-3]
	\arrow[from=5-3, to=5-4]
	\arrow[from=5-4, to=5-5]
	\arrow[from=4-4, to=5-4]
	\arrow[from=3-3, to=3-4]
	\arrow[from=3-4, to=4-4]
	\arrow[from=2-3, to=3-3]
	\arrow[from=4-5, to=5-5]
	\arrow[from=4-4, to=4-5]
	\arrow[from=4-1, to=4-2]
	\arrow[from=4-3, to=4-4]
	\arrow[from=3-1, to=3-2]
	\arrow[from=3-2, to=3-3]
	\arrow[from=2-2, to=3-2]
	\arrow[from=3-2, to=4-2]
	\arrow[from=3-3, to=4-3]
	\arrow[from=4-2, to=4-3]
	\arrow[from=4-2, to=5-2]
	\arrow[from=4-3, to=5-3]
		\arrow[phantom, "\usebox\pullback" , very near start, color=black,from=1-1, to=2-2]
		\arrow[phantom, "\usebox\pullback" , very near start, color=black,from=2-1, to=3-2]
		\arrow[phantom, "\usebox\pullback" , very near start, color=black,from=3-1, to=4-2]
		\arrow[phantom, "\usebox\pullback" , very near start, color=black,from=4-1, to=5-2]
		\arrow[phantom, "\usebox\pullback" , very near start, color=black,from=2-2, to=3-3]
		\arrow[phantom, "\usebox\pullback" , very near start, color=black,from=3-2, to=4-3]
		\arrow[phantom, "\usebox\pullback" , very near start, color=black,from=4-2, to=5-3]
		\arrow[phantom, "\usebox\pullback" , very near start, color=black,from=3-3, to=4-4]
		\arrow[phantom, "\usebox\pullback" , very near start, color=black,from=4-3, to=5-4]
		\arrow[phantom, "\usebox\pullback" , very near start, color=black,from=4-4, to=5-5]
\end{tikzcd}    
\]
\end{proof}

\begin{lemma}\label{lem:forget-connection}
Forgetting the connection data induces a natural homotopy commutative diagram
 \[
\begin{tikzcd}
    {\conn{\B \mathrm{String}^{(0)}}}\arrow[r]\arrow[d]&\conn{\mathbf{B}\mathrm{Spin}}\arrow[d]\\
    {\B \mathrm{String}}\arrow[r]&{\B \mathrm{Spin}}
\end{tikzcd}
 \]
\end{lemma}
\begin{proof}
From Definition \ref{def:geom-string-struct} and Remark \ref{rem:forget-conn} we have the homotopy commutative diagram
\[
\begin{tikzcd}
	{\conn{\B \mathrm{String}^{(0)}}} & \ast \\
	{\conn{\B \mathrm{Spin}}} & {\B^{3}\mathrm{U}(1)}\\
	{\B \mathrm{Spin}} &{\B^{3}\mathrm{U}(1)}
	\arrow[from=1-1, to=2-1]
	\arrow[from=2-1, to=2-2,"\frac{1}{2}\hat{\mathbf{p}}_1^{(0)}"]
	\arrow[from=1-2, to=2-2]
	\arrow[from=1-1, to=1-2]
	\arrow[from=2-1, to=3-1]
	\arrow[from=3-1, to=3-2,"\frac{1}{2}{\mathbf{p}}_1"]
	\arrow[equal,from=2-2, to=3-2]
		\arrow[phantom, "\usebox\pullback" , very near start, color=black,from=1-1, to=2-2]
\end{tikzcd}.
\]
Forgetting the middle line and recalling the definition of $\B \mathrm{String}$ and the universal property of homotopy pullbacks, we get a homotopy commutative diagram
\[
\begin{tikzcd}
	{\conn{\B \mathrm{String}^{(0)}}} & &\ast \\
	{\conn{\B \mathrm{Spin}}} & {\B \mathrm{String}} &\\
	{\B \mathrm{Spin}} & &{\B^{3}\mathrm{U}(1)}
	\arrow[from=1-1, to=2-1]
	\arrow[from=1-1, to=2-2]
	\arrow[from=1-3, to=3-3]
	\arrow[from=1-1, to=1-3]
	\arrow[from=2-1, to=3-1]
	\arrow[from=2-2, to=1-3, bend right=15]
	\arrow[from=2-2, to=3-1, bend right=15]
	\arrow[from=3-1, to=3-3,"\frac{1}{2}{\mathbf{p}}_1"]
		\arrow[phantom, "\usebox\pullback" , very near start, color=black,from=2-2, to=3-3]
\end{tikzcd}
\]
which in particular in its leftmost part gives the statement.
\end{proof}
\begin{proposition}\label{prop:geom-string-structures-exist}
    Let $M$ be a smooth manifold, and let $P\colon M\to \B \mathrm{Spin}$ be a principal spin bundle on $M$. Then $P$ can be enhanced to a geometric string structure on $M$ if and only if $\frac{1}{2}p_1(P)=0$.
\end{proposition}
\begin{proof}
By Lemma \ref{lem:forget-string} and Lemma \ref{lem:forget-connection}, a lift of $P$ to a geometric string structure induces in particular a lift of $P$ to a string bundle, i.e., to a morphism $\tilde{P}\colon M\to \B \mathrm{String}$. By definition of $\B \mathrm{String}$ this is equivalent to a trivialization of $\frac{1}{2}\mathbf{p}_1(P)$, and so implies $\frac{1}{2}p_1(P)=0$ in $H^4(M;\mathbb{Z})$. Vice versa, 
if $\frac{1}{2}p_1(P)=0$ in $H^4(M;\mathbb{Z})$ then $\frac{1}{2}\mathbf{p}_1(P)$ is homotopically trivial in $[M,\B^3\mathrm{U}(1)]$. Since every spin bundle admits a spin connection, we can lift $P$ to a principal spin connection $(P,\nabla)\colon M \to \conn{\B \mathrm{Spin}}$. By Remark \ref{rem:forget-conn}, we have $\frac{1}{2}\mathbf{p}_1^{(0)}(P,\nabla)\simeq \frac{1}{2}\mathbf{p}_1(P)$,  and so $\frac{1}{2}\mathbf{p}_1^{(0)}(P,\nabla)$ is homotopically trivial in $[M,\B^3\mathrm{U}(1)]$. By Definition \ref{def:geom-string-struct} this means that $(P,\nabla)$ can be lifted to a morphism $\eta^{(0)}\colon M\to \conn{\B \mathrm{String}}^{(0)}$. From the proof of Lemma \ref{lem:forget-string} we see that the obstruction to lifting $\eta^{(0)}$ to a morphism $\eta^{(1)}\colon M\to \conn{\B \mathrm{String}}^{(1)}$ is given by an element in $[M,\B^2\Omega^1]\cong H^2(M,\Omega^1)=0$, where the latter equality follows from the fact that $\Omega^1$ is a fine sheaf on $M$. Similarly, the obstruction to lifting $\eta^{(1)}$ to a morphism $\eta^{(2)}\colon M\to \conn{\B \mathrm{String}}^{(2)}$ is given by an element in $[M,\B\Omega^2]\cong H^1(M,\Omega^2)=0$. Therefore we see there is no obstruction to enhancing $P$ to a geometric string structure.
\end{proof}

\begin{remark}\label{rem:do-string-connections-exist?}
Since $H^0(M,\Omega^3)=\Omega^3(M)$ is nonzero, the argument in the proof of Proposition \ref{prop:geom-string-structures-exist} can't be used to show that a string bundle on $M$ always admit a string connection. At the same time, it does not really prevent such a connection from existing: the argument shows that for a particular choice of a geometric string structure lifting the topological structure of the string bundle, there could be a nonzero obstruction (the canonical 3-form of the geometric string structure) to an actual string connection. But it is still possible that changing the spin connection on the underlying spin bundle and its lift to a geometric string structure, this obstruction may vanish. The question whether a string bundle  always admits a string connection or not\footnote{Or, more generally, whether a principal $G$ bundle for $G$ a Lie 2-group always admits a $G$-connection as it is the case for ordinary Lie groups.} is, to our knowledge, undecided at the moment and there is no consensus on which answer should be expected, see, e.g., the discussion at \url{https://mathoverflow.net/q/426197}.
\end{remark}

\section{Morphisms of morphisms of abelian groups and homotopy fibers}
{\color{black}We collect here some technical facts on the essential fibers of monoidal functors associated with commutative diagrams of abelian groups, to be used in Section \ref{sec:bnr} where we prove our main result.}
In Section \ref{sec:cat-from-mor} we have seen how one can associate a rigid symmetric monoidal category to a morphism of abelian groups. This construction is actually functorial. To see this, recall that morphisms of abelian groups are the objects of a category {\color{black}$\mathsf{Arr}(\mathsf{Ab})$}, whose morphisms are commutative diagrams. That is, if $\varphi_H\colon H_{\mathrm{mor}} \to H_{\mathrm{ob}}$ and $\varphi_G\colon G_{\mathrm{mor}}\to G_{\mathrm{ob}}$ are objects of $\mathsf{Arr}(\mathsf{Ab})$, i.e., morphisms of abelian groups, then a morphism from $\varphi_H$ to $\varphi_G$ in $\mathsf{Arr}(\mathsf{Ab})$ is a commutative diagram of the form  
\[
\begin{tikzcd}
    H_{\mathrm{mor}}\arrow[r,"\phi_H"]\arrow[d,"f_{\mathrm{mor}}"']&H_{\mathrm{ob}}\arrow[d,"f_{\mathrm{ob}}"]\\
    G_{\mathrm{mor}}\arrow[r,"\phi_G"]&G_{\mathrm{ob}}
\end{tikzcd}
\]
The pair $(f_{\mathrm{ob}},f_{\mathrm{mor}})$ defines a symmetric monoidal functor $f\colon \phi_H^\otimes \to \phi_G^\otimes$, and one easily checks that the construction is functorial, {\color{black}i.e., it defines a functor\footnote{\color{black}Strictly speaking this is a 2-functor from the category $\mathsf{Arr}(\mathsf{Ab})$ seen as a 2-category with trivial 2-morphisms, to the 2-category $\mathsf{SMC}$. One could naturally enhance $\mathsf{Arr}(\mathsf{Ab})$ by adding to it nontrivial 2-morphisms/homotopies. We are not going to do this in full generality, but glimpses of this possibility appear in Lemma \ref{lem:hofiber-and-ker} and in Remark \ref{rem:hofiber-and-ker}.} 
\[
\mathsf{Arr}(\mathsf{Ab})\to \mathsf{SMC},
\]
where $\mathsf{SMC}$ is the 2-category of (small) symmetric monoidal categories with symmetric monoidal functors and symmetric monoidal natural transformations.}
We recall the following
\begin{definition}
Given a functor $p:\mathcal{D}\to \mathcal{C}$ and an object $c$ of $\mathcal{C}$, the \emph{homotopy fiber} (or  \emph{essential fiber}) of $p$ over $c$ is the category $\hofib{p;c}$ with objects the pairs $(x,b)$ with $x$ an object in $\mathcal{D}$ and $b\in \hom_{\mathcal{C}}(c,p(x))$ an isomorphism; morphisms from $(x,b)$ to $(x',b')$ in $\hofib{p;c}$ are those morphisms $a:x\to x'$ in $\mathcal{D}$ such that the diagram
\[
\begin{tikzcd}[sep= tiny, ampersand replacement=\&]
    p(x)\arrow[rr, "p(a)"]\& \&p(x')\\
    \& c \arrow[ul, "b"] \arrow[ur, "b'"']\&
\end{tikzcd}
\]
commutes.
By relaxing the condition that $b$ is an isomorphism, and allowing it to be an arbitrary morphism, we obtain the notion of \emph{lax homotopy fiber} and denote it by $\hofiblax{p;c}$.
\end{definition}
{\color{black}
\begin{remark}The use of the term ``homotopy fiber'' to denote the essential fiber of a functor can be given a rigorous justification in terms of the canonical model structure on the 2-category of all categories in a fixed topos, see \cite{joyal-tierney:Strong-Stacks-and-Classifying-Spaces}
\end{remark}
}
When $p:\mathcal{D}\to \mathcal{C}$ is a monoidal functor between monoidal categories, we will always take $c$ to be the monoidal unit $\mathbf{1}_\mathcal{C}$ of $\mathcal{C}$, and will simply write $\hofib{p}$ and $\hofiblax{p}$ for $\hofib{p;\mathbf{1}_\mathcal{C}}$ and $\hofiblax{p;\mathbf{1}_\mathcal{C}}$, respectively. It is immediate to see that if $p$ is a monoidal functor, the monoidal structures of $\mathcal{C}$ and $\mathcal{D}$ induce a natural monoidal category structure on $\hofib{p}$ and on $\hofiblax{p}$.

\begin{remark}
If $\mathcal{C}$ is a  groupoid, then there is no difference between the homotopy fiber of $p$ at $c$ and its lax homotopy fiber. This consideration in particular applies to the categories $\phi^\otimes$.
\end{remark}

For the monoidal functor $f\colon \phi_H^\otimes \to \phi_G^\otimes$ associated with a commutative diagram of abelian groups, we can give a simple explicit description of its homotopy fiber. This is provided by the following easy Lemma.
\begin{lemma}\label{lem:hofiber}
Let \[
\begin{tikzcd}
    H_{\mathrm{mor}}\arrow[r,"\phi_H"]\arrow[d,"f_{\mathrm{mor}}"']&H_{\mathrm{ob}}\arrow[d,"f_{\mathrm{ob}}"]\\
    G_{\mathrm{mor}}\arrow[r,"\phi_G"]&G_{\mathrm{ob}}
\end{tikzcd}
\]
be a commutative diagram of abelian groups, and let $f\colon \phi_H^\otimes \to \phi_G^\otimes$ be the associated monoidal functor. Then we have
\begin{align*}
    \mathrm{Ob}(\hofib{f})&=G_{\mathrm{mor}}\times_{G_{\mathrm{ob}}}H_{\mathrm{ob}}
    \\
    \mathrm{Mor}_{\hofib{f}}\left((g,h),(g',h')\right)&=\left\{ x\in H_{\mathrm{mor}}\,\text{ s.t. } \begin{cases} 
    f_{\mathrm{mor}}(x)=g'-g\\
    \phi_H(x)=h'-h
        \end{cases}\phantom{\int^\int_\int}\!\!\!\!\!\!
    \right\}.
\end{align*}
\end{lemma}
\begin{proof}
By definition, an object in $\hofib{f}$ is a pair $(g,h)$ with $h$ an object in $\phi_H$ and $g$ an isomorphism from $\mathbf{1}_{\varphi_G}$ to $f(h)$ in $\phi_G$. Making this explicit, we see that $h\in H_{\mathrm{ob}}$, and $g\in G_{\mathrm{mor}}$ is such that $0_{G_\mathrm{ob}}+\phi_G(g)=f_{\mathrm{ob}}(h)$, i.e., $\phi_G(g)=f_{\mathrm{ob}}(h)$. This gives $\mathrm{Ob}(\hofib{f})=G_{\mathrm{mor}}\times_{G_{\mathrm{ob}}}H_{\mathrm{ob}}$. As far as concerns the morphisms, we have
\begin{align*}
\mathrm{Mor}_{\hofib{f}}\bigl((g,h),&(g',h')\bigr)=\left\{x:h\to h'\, \text{s.t.} 
\begin{tikzcd}[sep= tiny, ampersand replacement=\&]
    f_{\mathrm{ob}}(h)\arrow[rr, "f_{\mathrm{mor}}(x)"]\& \&f_{\mathrm{ob}}(h')\\
    \& 0_{G_\mathrm{ob}} \arrow[ul, "g"] \arrow[ur, "g'"']\&
\end{tikzcd}\text{ commutes}
\right\}\\
=&\lrb{x\in H_{\mathrm{mor}}\text{ s.t. }   f_{\mathrm{mor}}(x)+g=g'  \text{ and }h+\phi_H(x)=h' }\\
=&\lrb{x\in H_{\mathrm{mor}}\text{ s.t. } f_{\mathrm{mor}}(x)=g'-g  \text{ and } \phi_H(x)=h'-h }.
\end{align*}
\end{proof}

\begin{lemma}\label{lem:hofiber-and-ker}
A commutative diagram of abelian groups of the form
\[
\begin{tikzcd}
    H_{\mathrm{mor}}\arrow[r,"\phi_H"]\arrow[d,"f_{\mathrm{mor}}"']&H_{\mathrm{ob}}\arrow[d,"f_{\mathrm{ob}}"]\arrow[dl,"\lambda"']\\
    G_{\mathrm{mor}}\arrow[r,"\phi_G"]&G_{\mathrm{ob}}
\end{tikzcd}
\]
induces a symmetric monoidal functor
$\Xi_\lambda\colon\hofib{f}\to \ker(\phi_G)^\otimes$ acting on the objects as
$(g,h)\mapsto g-\lambda(h)$. Moreover, $\Xi_\lambda$ is an equivalence iff $\phi_H$ is an isomorphism. 
\end{lemma}
{\color{black}The last line in the statement of Lemma \ref{lem:hofiber-and-ker} may be a bit surprising at first, since $\Xi_\lambda$ depends on $\lambda$ whereas $\varphi_H$ does not. But actually, when $\varphi_H$ is an isomorphism, then $\lambda$ is unique, given by $f_{\mathrm{mor}}\circ \varphi_H^{-1}$.}  
\begin{proof}
To begin with, let us check that $\Xi_\lambda$ does indeed take its values in $\ker(\phi_G)^\otimes$. For $(g,h)$ an object in $\hofib{f}$, by Lemma \ref{lem:hofiber} we have
\[
\phi_G(\Xi_\lambda(g,h))=
\phi_G(g-\lambda(h))=\phi_G(g)-f_{\mathrm{ob}}(h)=0_{G_{\mathrm{ob}}}.
\]
Since  {\color{black}the only morphisms of $\ker(\phi_G)^\otimes$ are identities}, in order to show that $\Xi_\lambda$ is a functor we only need to show that a morphism
$x:(g,h)\to(g',h')$ in $\hofib{f}$ induces an identity $\Xi_\lambda(g,h)=\Xi_\lambda(g',h')$. From Lemma \ref{lem:hofiber} we know that $x$ satisfies the two identities $f_{\mathrm{mor}}(x)=g'-g$  and $\phi_H(x)=h'-h$. Since $f_{\mathrm{mor}}=\lambda\circ \phi_H$, the two identities together give \[
g'-g=f_{\mathrm{mor}}(x)=\lambda\left(\phi_H(x)\right)=\lambda(h'-h)=\lambda(h')-\lambda(h),
\]
i.e., $g-\lambda(h)=g'-\lambda(h')$
which is the desired identity. The monoidality of $\Xi_\lambda$ is manifest. 

To prove the second part of the statement, we begin by noticing that since $(g,0_{H_\mathrm{ob}})$ is mapped to $g$, the functor $\Xi_\lambda$ is always essentially surjective and so it is an equivalence iff it is fully faithful.
If $\Xi_\lambda$ is fully faithful, then we must have
\[
\mathrm{Mor}_{\hofib{f}}\bigl((g,h),(g,h)\bigr)=\{0_{H_\mathrm{mor}}\}.
\]
Since \[
\ker(\phi_H)\subseteq \ker(f_{\mathrm{mor}})\subseteq \mathrm{Mor}_{\hofib{f}}\bigl((g,h),(g,h)\bigr),
\]
this implies the injectivity of $\phi_H$. Since both $(0_{G_\mathrm{mor}}, 0_{H_\mathrm{ob}})$ and $(\lambda(h),h)$ are mapped by $\Xi_\lambda$ to $0_{G_\mathrm{mor}}$, fully faithfulness implies the existence of a morphism from $(0_{G_\mathrm{mor}}, 0_{H_\mathrm{ob}})$ to $(\lambda(h),h)$ in $\hofib{f}$, and by  Lemma \ref{lem:hofiber} such a morphism is in particular an element $x$ in $H_\mathrm{mor}$ such $\phi_H(x)=h$. This shows that $\phi_H$ is also surjective, and so an isomorphism. Vice versa, assume $\phi_H$ is an isomorphism, and let $(g,h)$ and $(g',h')$ be such that $\Xi_\lambda(g,h)=\Xi_\lambda(g',h')$. This is equivalent to $\lambda(h'-h)=g'-g$. Since  $\phi_H$ is an isomorphism there exists a unique element $x$ in $H_\mathrm{mor}$ such that $\phi_H(x)=h'-h$. We also have
\[
f_\mathrm{mor}(x)=\lambda\left(\phi_H(x)\right)=\lambda(h'-h)=g'-g
\]
so that this element $x$ is a morphism (necessarily unique) from $(g,h)$ to $(g',h')$ in $\hofib{f}$. Finally, let $(g,h)$ and $(g',h')$ be such that $\Xi_\lambda(g,h)\neq \Xi_\lambda(g',h')$. If  $\mathrm{Mor}_{\hofib{f}}\bigl((g,h),(g',h')\bigr)$ is nonempty, then also $\mathrm{Mor}_{\ker(\phi_G)^\otimes}(g-\lambda(h),g'-\lambda(h'))$ has to be nonempty, and so we get $g-\lambda(h)=g'-\lambda(h')$, a contradiction. This shows that if $\phi_H$ is an isomorphism then $\Xi_\lambda$ is an equivalence.
\end{proof}
{\color{black}
\begin{remark}\label{rem:hofiber-and-ker}
A more conceptual understanding of Lemma \ref{lem:hofiber-and-ker} is as follows. The category of morphism of abelian groups can be equivalently seen as the category of chain complexes of $\mathbb{Z}$-modules concentrated in degree $0$ and $1$ (with $A_{\mathrm{ob}}$ in degree 0 and $A_{\mathrm{mor}}$ in degree 1). In this context, the map $\lambda$ is a homotopy between the morphism of chain complexes $f=(f_{\mathrm{ob}},f_{\mathrm{mor}})$ and the zero morphism. The homotopy fiber of $f$ will then be equivalent to the homotopy fiber of the zero morphism, i.e., to the direct sum $H_\bullet\oplus\mathbf{\Omega}G_\bullet$, where $\mathbf{\Omega}G_\bullet$ is the loop space object of $G_\bullet$ in the category of chain complexes concentrated in degrees $0$ and $1$ (in the standard, i.e., projective, model structure). A representative for $\mathbf{\Omega}G_\bullet$ is nothing but the chain complex consisting of $\ker (\varphi_G)$ concentrated in degree $0$. Finally, the summand $H_\bullet$ is homotopically trivial iff it is acyclic, and so iff $\varphi_H$ is an isomorphism. 
\end{remark}}
\section{The Bunke--Naumann--Redden morphism}\label{sec:bnr}
{
{\color{black}
We can now show how the constructions and examples presented so far conspire to reproduce the Bunke--Naumann--Redden morphism from the Introduction.}
Since $\conn{\B \mathrm{String}}^{(2)}$ comes equipped with a morphism of smooth stacks
$
\conn{\B \mathrm{String}}^{(2)}\xrightarrow{\omega_3}\Omega^3
$
we have an associated symmetric monoidal functor $\Bord_{4,3}^{\mathrm{or}}(\conn{\B \mathrm{String}}^{(2)})\to \Bord_{4,3}^{\mathrm{or}}(\Omega^{3})$. Composing this with the $\mathrm{id}_{\mathbb{R}}^\otimes$-valued TQFT from Example \ref{ex:Stokes}, we obtain an $\mathrm{id}_{\mathbb{R}}^\otimes$-valued TQFT
\[
Z_{\mathrm{String}}\colon \Bord_{4,3}^{\mathrm{or}}(\conn{\B \mathrm{String}}^{(2)})\to \mathrm{id}_{\mathbb{R}}^\otimes.
\]
It maps a closed oriented 3-manifold $M_3$ equipped with a geometric string structure $\eta^{}_{M_3}$ to $\int_{M_3}\omega^\eta_{3;M_3}$, where $\omega^\eta_{3;M_3}$is the 3-form on $M_3$ associated with the geometric string structure, and it maps an oriented 4-manifold $W_4$ equipped with a geometric string structure to 
\[
\int_{W_4}d\omega^\eta_{3;W_4}=\frac{1}{2}\int_{W_4}\mathbf{p}_1^{\mathrm{CW}}(\nabla^{}_{W_4}),
\]
where the last identity is Remark \ref{rem:omega3-and-p1}. 

From Example \ref {ex:tqft-spin-conn} we have an $\exp(2\pi i -)^\otimes$-valued TQFT \[
Z_{\mathrm{Spin}}\colon \Bord_{4,3}^{\mathrm{or}}(\conn{\mathbf{B}\mathrm{Spin}})\to \exp(2\pi i -)^\otimes
\]
mapping 
a closed oriented 3-manifold $M_3$ equipped with a spin connection $(P_{M_3},\nabla^{}_{M_3})$  to $\mathrm{hol}_{M_3}\left(\frac{1}{2}\hat{\mathbf{p}}_1(\nabla^{}_{M_3})\right)$ and an oriented 4-manifold $W_4$ (equipped with a spin connection) to $\frac{1}{2}\int_{W_{4}}\mathbf{p}_1^{\mathrm{CW}}(\nabla^{}_{W_4})$. Moreover, the projection $\pi^{\mathrm{String}}_{\mathrm{Spin}}\colon\conn{\B \mathrm{String}}^{(2)}\to \conn{\B \mathrm{Spin}}$ induces a symmetric monoidal functor 
\[
\Bord_{4,3}^{\mathrm{or}}(\pi^{\mathrm{String}}_{\mathrm{Spin}})\colon \Bord_{4,3}^{\mathrm{or}}(\conn{\B \mathrm{String}}^{(2)})\to \Bord_{4,3}^{\mathrm{or}}(\conn{\B \mathrm{Spin}})
\]
and the commutative diagram of abelian groups
\[
\begin{tikzcd}
	\R && \R \\
	\R && {\mathrm{U}(1)}
	\arrow[from=1-1, to=1-3, "\mathrm{id}_\R"]
	\arrow[from=1-1, to=2-1, "\mathrm{id}_\R"']
	\arrow[from=2-1, to=2-3, "\exp(2\pi i-)"]
	\arrow[from=1-3, to=2-3, "\exp(2\pi i-)"]
\end{tikzcd}
\]
induces a symmetric monoidal functor 
\[
(\mathrm{id}_\R,\exp(2\pi i -))\colon\mathrm{id}_\R^\otimes \to\exp(2\pi i -)^\otimes,
\]
leading to the following result.
\begin{lemma}\label{lem:comm-diagram-symm-mon}
The diagram of symmetric monoidal functors
\[
\begin{tikzcd}
	\Bord_{4,3}^{\mathrm{or}}(\conn{\B \mathrm{String}}^{(2)}) && \mathrm{id}_\R^\otimes \\
	\Bord_{4,3}^{\mathrm{or}}(\conn{\B \mathrm{Spin}}) && {\exp(2\pi i -)^\otimes}
	\arrow[from=1-1, to=1-3, "Z_{\mathrm{String}}"]
	\arrow[from=1-1, to=2-1, "\Bord_{4,3}^{\mathrm{or}}(\pi^{\mathrm{String}}_{\mathrm{Spin}})"']
	\arrow[from=2-1, to=2-3, "Z_{\mathrm{Spin}}"]
	\arrow[from=1-3, to=2-3, "{(\mathrm{id}_\R,\exp(2\pi i -))}"]
\end{tikzcd}
\]
commutes, with identity 2-cell.
\end{lemma}
\begin{proof}
The commutativity of the diagram at the objects level is the identity
\[
\exp\biggl(2\pi i \int_{M_3}\omega^\eta_{3;M}\biggr)=\mathrm{hol}_{M_3}\left(\frac{1}{2}\hat{\mathbf{p}}_1(\nabla^{}_{M_3})\right)
\]
for a closed 3-manifold $M_3$ equipped with a geometric string structure. This is provided by Remark \ref{rem:curvature-and-deRham} {\color{black} together with the commutative diagram
\[
\begin{tikzcd}
	{[M_3,\conn{\B \mathrm{String}}^{(2)})]} && {[M_3,\conn{\B \mathrm{Spin}}]} \\
	{[M_3,\Omega^3]} && {[M_3,\mathbf{B}^3\conn{\mathrm{U}(1)}]}
	\arrow[from=1-1, to=1-3, ""]
	\arrow[from=1-1, to=2-1, "\omega_3"']
	\arrow[from=2-1, to=2-3, ""]
	\arrow[from=1-3, to=2-3, "\frac{1}{2}\hat{\mathbf{p}}_1"]
\end{tikzcd}
\]
coming from the definition of $\conn{\B \mathrm{String}}^{(2)}$}. {\color{black} The commutativity of the diagram at the morphisms level is the identity 
\[
\int_{W_4}d\omega^\eta_{3;W_4}=\frac{1}{2}\int_{W_4}\mathbf{p}_1^{\mathrm{CW}}(\nabla^{}_{W_4})
\]
recalled above.}
\end{proof}
By the naturality of the (lax) hofiber construction, and recalling that for a morphism of groupoids there is no distinction between the homotopy fiber and the lax homotopy fiber, Lemma \ref{lem:comm-diagram-symm-mon} implies we have a distinguished symmetric monoidal functor
\[
\hofiblax{\Bord_{4,3}^{\mathrm{or}}(\pi^{\mathrm{String}}_{\mathrm{Spin}})}\to\hofib{\mathrm{id}_\R,\exp(2\pi i -)}.
\]
Since the commutative diagram defining the morphism $(\mathrm{id}_\R,\exp(2\pi i -))$ factors as
\[
\begin{tikzcd}
	\R && \R \\
	\R && {\mathrm{U}(1)}
	\arrow[from=1-1, to=1-3, "\mathrm{id}_\R"]
	\arrow[from=1-1, to=2-1, "\mathrm{id}_\R"']
	\arrow[from=2-1, to=2-3, "\exp(2\pi i-)"']
	\arrow[from=1-3, to=2-3, "\exp(2\pi i-)"]
	\arrow[from=1-3, to=2-1, "\mathrm{id}_\R"']
\end{tikzcd},
\]
from Lemma \ref{lem:hofiber-and-ker} we have a symmetric monoidal equivalence 
\[
\Xi_{\mathrm{id}_\R}\colon \hofib{(\mathrm{id}_\R,\exp(2\pi i -))}\to \ker(\exp(2\pi i-)^\otimes) =\Z^\otimes,
\]
acting on $\mathrm{Ob}(\hofib{(\mathrm{id}_\R,\exp(2\pi i -))})=\R\times_{\mathrm{U}(1)}\R$ as
$(g,h)\mapsto g-h$. Putting everything together we obtain a symmetric monoidal functor
\[
Z^{\mathrm{String}}_{\mathrm{Spin}}\colon \hofiblax{\Bord_{4,3}^{\mathrm{or}}(\pi^{\mathrm{String}}_{\mathrm{Spin}})} \to \Z^\otimes
\]

\begin{proposition}\label{prop:main}
The symmetric monoidal functor $Z^{\mathrm{String}}_{\mathrm{Spin}}$ is the Bunke--Naumann--Redden map $\psi$ from the Introduction.
\end{proposition}
\begin{proof}
 By the definition of lax homotopy fiber, an object in the source of $Z^{\mathrm{String}}_{\mathrm{Spin}}$ consists of an object in $\Bord_{4,3}^{\mathrm{or}}(\conn{\B \mathrm{String}}^{(2)})$, i.e., in a closed oriented 3-manifold $M_3$ equipped with a geometric string structure $\eta^{}_{M_3}$, together with a morphism in $\Bord_{4,3}^{\mathrm{or}}(\conn{\B \mathrm{Spin}})$ from the unit object (i.e., the empty manifold) to the spin manifold underlying $(M_3,\eta^{}_{M_3})$. This is precisely the datum of a spin 4-manifold $W_4$ with $\partial W_4=M_3$, equipped with a spin connection $\nabla^{}_{W_4}$ such that the restriction $\nabla^{}_{W_4}\bigr\vert_{M_3}$ coincides with the spin connection datum of the geometric string structure $\eta^{}_{M_3}$. In other words, the objects of $\hofiblax{\Bord_{4,3}^{\mathrm{or}}(\conn{\B \mathrm{String}}^{(2)})\to \Bord_{4,3}^{\mathrm{or}}(\conn{\B \mathrm{Spin}})}$ are precisely the geometric data of the Bunke--Naumann--Redden construction. Now, let us compute $Z^{\mathrm{String}}_{\mathrm{Spin}}(M_3,\eta^{}_{M_3},W_4,\nabla^{}_{W_4})$, where we see the quadruple $(M_3,\eta^{}_{M_3},W_4,\nabla^{}_{W_4})$ as a pair consisting of the object $(M_3,\eta^{}_{M_3})$ and of the morphism $(W_4,\nabla^{}_{W_4})$. On the object $(M_3,\eta^{}_{M_3})$ we act with $Z_{\mathrm{String}}$ at the object level, obtaining the object $\int_{M_3}\omega^\eta_{3;M_3}$ of $\mathrm{id}_\R^\otimes$. On the morphism $(W_4,\nabla^{}_{W_4})$ we act with $Z_{\mathrm{Spin}}$ at the morphism level, obtaining the morphism $\frac{1}{2}\int_{W_{4}}\mathbf{p}_1^{\mathrm{CW}}(\nabla^{}_{W_4})$ of $\exp(2\pi i -)^\otimes$. We thus obtain the object $(\frac{1}{2}\int_{W_{4}}\mathbf{p}_1^{\mathrm{CW}}(\nabla^{}_{W_4}), \int_{M_3}\omega^\eta_{3;M_3})$  of $\hofib{(\mathrm{id}_\R,\exp(2\pi i -))}$. The equivalence $\Xi$ finally maps this to
 \[
 \psi(M_3,\eta^{}_{M_3},W_4,\nabla^{}_{W_4})=\frac{1}{2}\int_{W_{4}}\mathbf{p}_1^{\mathrm{CW}}(\nabla^{}_{W_4})-\int_{M_3}\omega^\eta_{3;M_3}.
 \]
 Notice how by construction $\psi$ takes integral values.
\end{proof}
\begin{remark}[Further examples]\label{rem:further-examples}
There are other situations where exactly the same construction as the one we presented here applies. For instance, one can replace the first fractional Pontryagin class $\frac{1}{2}p_1$ with the first Chern class $c_1$. In this case, one notices that the classifying space $\mathrm{BSU}$ of the special unitary group is the homotopy fiber of $c_1\colon BU\to K(\Z,2)$ so that $\mathrm{SU}$ and $\mathrm{U}$ enjoy the same kind of relationship as $\mathrm{String}$ and $\mathrm{Spin}$. Then the construction presented in the article is repeated verbatim to obtain an integral formula realizing the isomorphism
\[
 \Bord_1^{\mathrm{SU}}\cong \Z/2\Z.
\]
In this case the relevant vanishing bordism group is $\Bord_1^{\mathrm{U}}=0$, and the relevant index theorem is the Gauss-Bonnet formula: for a closed oriented 2-manifold $W_2$ the integer $\int_{W_2}c_1(W_2)$ is even.
\par
Another example is obtained by going one step higher in the Whitehead tower of the orthogonal group and by considering the $\mathrm{Fivebrane}$ group instead of the $\mathrm{String}$ group, {\color{black}see, e.g., \cite{botvinnik, sati-fivebrane}}. One has that the classifying space $\mathrm{BFivebrane}$ is the homotopy fiber of $\frac{1}{6}p_2\colon \mathrm{BString}\to K(\Z,8)$, so $\mathrm{Fivebrane}$ is to $\mathrm{String}$ as $\mathrm{String}$ is to $\mathrm{Spin}$, and one has a vanishing 7-dimensional bordism group given by $\Bord_7^{\mathrm{String}}=0$. For 
a closed oriented string 8-manifold $W_8$, which is in particular a closed oriented spin 8-manifold with $p_1(W_8)=0$ in $H^8(W_8;\mathbb{Q})$, the Atiyah--Singer index theorem gives that
\[
\hat{A}(W_8)=\frac{1}{5760}\int_{W_8}(-4p_2(W_8)+7p_1(W_8)^2)=-\frac{1}{1440}\int_{W_8}p_2(W_8)
\]
is an integer. Therefore $\frac{1}{6}\int_{W_8}p_2(W_8)\in 240 \Z$, and verbatim repeating the arguments in the present paper should give an integral formula realizing the isomorphism
\[
 \Bord_7^{\mathrm{Fivebrane}}\cong \Z/240\Z.
\]
The reader should however be advised that the argument for fivebrane bordism sketched above is not complete: in order to make it fully work, one would need to know that any principal string bundle admits a string connection and as we noticed in Remark \ref{rem:do-string-connections-exist?} this is presently not clear. The fact that the argument presented above produces the correct order $240$ for the seventh fivebrane bordism group could be seen as an indication that the answer to whether every principal string bundle admits a string connection may be affirmative.
\end{remark}

}

\section{Conflict of Interest statement}
 On behalf of all authors, the corresponding author states that there is no conflict of interest.

\nocite{*}
\bibliographystyle{alpha}
\bibliography{bibliography}
\end{document}